\documentclass[11pt,reqno, a4paper]{amsart}
\usepackage{mathrsfs}
\usepackage{amsgen}
\usepackage{amscd}
\usepackage{booktabs}
\allowdisplaybreaks
\usepackage{amsmath}
\usepackage{latexsym}
\usepackage{amsfonts}
\usepackage{amssymb}
\usepackage{amsthm}
\usepackage{graphicx,epsfig}
\usepackage{color}
\usepackage{amsthm,amssymb}
\usepackage{graphicx}
\usepackage{sidecap}
\usepackage{enumerate}
\usepackage{amssymb,amsmath,graphicx,amsfonts,euscript}
\usepackage{color}
\usepackage{mathrsfs}
\usepackage{empheq}
\usepackage{cases}
\usepackage{cite}
\usepackage{ulem}
\usepackage{cancel}
\usepackage[margin=1.1in]{geometry}
\DeclareGraphicsExtensions{.eps}
\usepackage{amssymb,amsmath,graphicx,amsfonts,euscript,mathrsfs}
\usepackage{hyperref}

\usepackage{titlesec} 
\titleformat{\section}{\vskip10pt\large\bfseries}{\thesection.}{0.5em}{\centering\vspace{5pt}}
\titleformat{\subsection}{\vskip10pt\normalsize\bfseries}{\thesubsection.}{0.5em}{}

\newtheorem{theorem}{Theorem}[section]
\newtheorem{lemma}[theorem]{Lemma}
\newtheorem{proposition}[theorem]{Proposition}
\newtheorem{remark}[theorem]{Remark}
\theoremstyle{definition}

\def\R{\mathbb{R}}
\def\Z{\mathbb{Z}}
\def\T{\mathbb{T}}
\def\C{\mathbb{C}}

\def\l{\langle}
\def\r{\rangle}

\newcommand{\fe}{\mathrm{e}}

\newcommand{\bZ}{{\mathbb Z}}

\numberwithin{equation}{section}

\begin{document}

\title[]{A modified splitting method for the cubic nonlinear Schr\"odinger equation}

\author[]{\,\,Yifei Wu}
\address{\hspace*{-12pt}Yifei Wu: Center for Applied Mathematics, Tianjin University, 300072, Tianjin, P. R. China.}
\email{yerfmath@gmail.com}




\maketitle

\begin{abstract}\noindent
As a classical time-stepping method, it is well-known that the Strang splitting method reaches the first-order accuracy  by losing two spatial derivatives.

In this paper, we propose a modified splitting method for the 1D cubic nonlinear Schr\"odinger equation: 
\begin{align*} 
u^{n+1}=\fe^{i\frac\tau2\partial_x^2}\widetilde{\mathcal N}_\tau
\left[\fe^{i\frac\tau2\partial_x^2}\big(\Pi_\tau +\fe^{-2\pi i\lambda M_0\tau}\Pi^\tau \big)u^n\right],
\end{align*}
with 
$\widetilde{\mathcal N}_t(\phi)=\fe^{-i\lambda t|\Pi_\tau\phi|^2}\phi,$ and $M_0$ is the mass of the initial data.  

Suitably choosing the filters $\Pi_\tau$ and $\Pi^\tau$, it is shown rigorously that it  reaches the first-order accuracy  by only losing  $\frac32$-spatial derivatives. 
Moreover, if $\gamma\in (0,1)$,  the new method presents the convergence rate of $\tau^{\frac{4\gamma}{4+\gamma}}$ in $L^2$-norm for the $H^\gamma$-data; if $\gamma\in [1,2]$,  it presents the convergence rate of $\tau^{\frac25(1+\gamma)-}$ in $L^2$-norm for the $H^\gamma$-data.
These results are better than the expected ones for the standard  (filtered)  Strang  splitting methods.
Moreover, the mass is conserved:  
$$
 \frac1{2\pi}\int_\T |u^n(x)|^2\,d x\equiv M_0, \quad n=0,1,\ldots, L.
$$

The key idea is based on the observation that   the low frequency and high frequency components of solutions are almost separated (up to some smooth components). Then  the algorithm is constructed by tracking the solution behavior at  the low  and high frequency components separately.

\end{abstract}


\section{Introduction}\label{sec:introduction}

This article concerns the numerical solution of the cubic nonlinear Schr\"odinger (NLS) equation  
\begin{equation}\label{model}
 \left\{\begin{aligned}
& i\partial_tu(t,x)
=\partial_{xx} u(t,x)+\lambda|u(t,x)|^2u(t,x)
 &&\mbox{for}\,\,\, x\in\T\,\,\,\mbox{and}\,\,\, t\in(0,T] , \\
 &u(0,x)=u^0(x) &&\mbox{for}\,\,\,  x\in\T,
 \end{aligned}\right.
\end{equation}
on the one-dimensional torus $\T=(-\pi,\pi)$ with the initial value $u^0\in H^\gamma(\T),\gamma\ge 0$, where $\lambda=1$ and $-1$ are referred to as the focusing and defocusing cases, respectively. It is known that problem \eqref{model} is globally well-posed in $H^\gamma(\T)$ for $\gamma\ge 0$; see \cite{Bo}.

The $L^2$-solution of the NLS equation satisfies the following mass conservation law:
\begin{align}
\frac1{2\pi}\int_\T |u(t,x)|^2\,d x = \frac1{2\pi}\int_\T |u^0(x)|^2\,d x\triangleq M_0  \quad\mbox{for}\,\,\, t>0 .  \label{mass}
\end{align}





We denote  the flow $\mathcal N_t$ to be 
$$
\mathcal N_t(\phi)=\fe^{-i\lambda t|\phi|^2}\phi,
$$
which solves the equation 
\begin{equation*}
 \left\{\begin{aligned}
& i\partial_tu(t,x)
=\lambda|u(t,x)|^2u(t,x),
 && \\
 &u(0,x)=\phi(x). &&
 \end{aligned}\right.
\end{equation*}
Moreover, denote $\fe^{i\tau\partial_x^2}$ to be the linear flow.

There are two classical  time-stepping method, namely, the Lie splitting method:
$$
\left[\fe^{i\tau\partial_x^2}\mathcal N_\tau\right]^n;
$$
and Strang splitting method:
$$
\left[\fe^{i\frac\tau2\partial_x^2}\mathcal N_\tau\fe^{i\frac\tau2\partial_x^2}\right]^n.
$$
See Strang \cite{Strang-1968} and the references therein. 

The traditional regularity assumption for the NLS equation for the time-stepping method, including the Strang splitting method and the Lie splitting method,  to have the first-order convergence in $H^{\gamma}(\T^d)$ is $u^0\in H^{\gamma+2}(\T^d)$ for $\gamma\ge 0$ (losing two derivatives).  In particular, Besse, Bid\'egaray, and Descombes \cite{besse} proved the statement for globally Lipschitz-continuous nonlinearities. Later, Lubich \cite{Lubich-2008} firstly analyzed the Strang method in the cubic nonlinearity case and also the Schr\"odinger-Poisson equation.  Further, 
Ignat \cite{Ignat-2011} (see also \cite{IgnatZua-2009}) originally introduced a filtered splitting method which can prove the statement for the power type nonlinearity: $|u|^pu, 1\le p<\frac4d, d\le 3$.  Recently, Ostermann, Rousset and Schratz \cite{ORS-2022}  proved that for 1D cubic NLS,  the filtered Lie splitting has $\frac\tau2$-order convergence in $L^2$ for the $H^{\gamma}(\T)$ initial data when $\gamma\in (0,2]$.

\subsection{Our main result}
To present  the new method, we define the filters which were introduced in \cite{Ignat-2011,IgnatZua-2009}. We denote by $\Pi_{\le N}: L^2(\T)\rightarrow L^2(\T)$ the low frequency projection operator defined by 
$$
\mathcal{F}_k[\Pi_{\le N} f]=
\left\{
\begin{aligned}
&\hat f_k &&\mbox{for}\,\,\, |k|\le N, \\
&0 &&\mbox{for}\,\,\, |k|> N ,
\end{aligned}
\right.
$$
and $\Pi_{> N}=\mathbb I-\Pi_{\le N}$. Here we denote by $\mathcal{F}_k[f]$ and $\hat f_k$ the $k$th Fourier coefficient of the function $f$.

Previously, the authors in  \cite{Ignat-2011,IgnatZua-2009} projected into the low frequencies to assure stability and guarantee convergence, and proposed the filtered splitting method:
$$
\left[\Pi_{\le N}\fe^{i\frac\tau2\partial_x^2}\mathcal N_\tau\>\fe^{i\frac\tau2\partial_x^2}\Pi_{\le N} \right]^n,\quad 
\mbox{with } N=\tau^{-\frac12}.
$$

Our main observation is that the solution at the low frequency and high frequency  has almost independent behaviors. 
Indeed,  by the structural Lemma \ref{lem:tri-est} below,  we can split the equation \eqref{model} into  a coupled system as follows:
\begin{equation}\label{model}
 \left\{\begin{aligned}
& i\partial_tu_{\le N}
=\partial_{xx} u_{\le N}+\lambda|u_{\le N}|^2u_{\le N}+\mathcal R_S^1(u_{\le N},u_{> N}),\\
&  i\partial_tu_{> N}
=\partial_{xx} u_{> N}+4\pi \lambda M_0 u_{> N}+\mathcal R_S^2(u_{\le N},u_{> N}),
 \end{aligned}\right.
\end{equation}
where $u_{\le N}=\Pi_{\le N} u, u_{> N}=\Pi_{> N} u$, and the coupled terms  $\mathcal R_S^1(u_{\le N},u_{> N})$ and $\mathcal R_S^2(u_{\le N},u_{> N})$ are smoother than the original nonlinearity. 
 
Dropping the smooth terms $\mathcal R_S^1$ and $\mathcal R_S^2$,  one may find that  the low frequency and high frequency components are separated. 
This enlightens us to track the low and the high frequency components separately, and  different splitting schemes
are proposed accordingly:  
\begin{align}\label{low-high-scheme}
\Pi_{\le N}\fe^{i\frac\tau2\partial_x^2}\mathcal N_\tau\big[\fe^{i\frac\tau2\partial_x^2}\Pi_{\le N}\big]^n;\quad 
\left[\Pi_{> N}\fe^{i\tau\partial_x^2-4\pi i\lambda M_0\tau}\right]^n.
\end{align}
Note that the two schemes defined in \eqref{low-high-scheme} are  mutually independent, and  approximate the $\Pi_{>N}$-part and $\Pi_{\le N}$-part of the truth solution, respectively. 
After suitably choosing the parameter $N$, we obtain a better error estimate than the standard (filtered) Lie or Strang  splitting method. 

%

More precisely,  define 
\begin{align}
\Pi_\tau= \Pi_{\le N}, \mbox{ with } \left\{
\begin{aligned}
&N= \tau^{-\frac{2}{4+\gamma}}, &\quad \mbox{for}\,\,\, \gamma\in (0,1), \\
&N= \tau^{-\frac25},\quad  &\quad\mbox{for}\,\,\,\gamma\in [1,2];
\end{aligned}
\right.
\qquad 
\Pi^\tau=\mathbb I-\Pi_\tau.
\end{align}
For any positive integer $L$, let $t_n=n\tau$, $n=0,1,\dots,L$, be a partition of the time interval $[0,T]$ with stepsize $\tau=T/L$. Then the splitting schemes we propose are 
\begin{align}\label{NuSo-NLS-1}
u^{n+1}=\Pi^\tau\big[\fe^{i\tau\partial_x^2-4\pi i\lambda M_0\tau}u^n\big]+\Pi_\tau\fe^{i\frac\tau2\partial_x^2}\mathcal N_\tau\big[\fe^{i\frac\tau2\partial_x^2}\Pi_\tau u^n\big], \quad n=0,1,\ldots, L - 1.
\end{align}

Furthermore, since 
$$
\fe^{i\frac\tau2\partial_x^2-i\lambda \tau\big|\fe^{i\frac\tau2\partial_x^2}\Pi_\tau u^n\big|^2+2\pi i\lambda M_0\tau}\fe^{i\frac\tau2\partial_x^2}\Pi^\tau u^n,
\quad \mbox{and }\quad 
\Pi^\tau\fe^{i\frac\tau2\partial_x^2-i\lambda \tau \big|\fe^{i\frac\tau2\partial_x^2}\Pi_\tau u^n\big|^2}\fe^{i\frac\tau2\partial_x^2}\Pi_\tau u^n
$$
are smooth terms (similar as $\mathcal R_S$ above), we can slightly modify \eqref{NuSo-NLS-1} and define  a new scheme as 
\begin{align} \label{NuSo-NLS-2}
u^{n+1}=\fe^{i\frac\tau2\partial_x^2}\widetilde{\mathcal N}_\tau
\left[\fe^{i\frac\tau2\partial_x^2}\big(\Pi_\tau +\fe^{-2\pi i\lambda M_0\tau}\Pi^\tau \big)u^n\right], \quad n=0,1,\ldots, L - 1,
\end{align}
where  the flow $\widetilde{\mathcal N}_t$ is defined by  
$$
\widetilde{\mathcal N}_t(\phi)=\fe^{-i\lambda t|\Pi_\tau\phi|^2}\phi.
$$


For any $a\in\R$,  we denote $a\pm=a\pm\epsilon$ for arbitrary small $\epsilon>0$.  
The main theoretical result of this paper is the following theorem.

\begin{theorem}\label{the:main}
Let $\gamma\in (0,2]$. If $u^0\in H^\gamma(\T)$, then there exist positive constants $\tau_0$  and $C$ such that for $\tau\leq\tau_0$ and $\tau N\le 1$, the numerical solution given by \eqref{NuSo-NLS-1} or \eqref{NuSo-NLS-2} has the following error bound:
\begin{itemize}
\item[(1)] If $\gamma\in (0,1)$,  then 
\begin{equation}\label{error-tau-1}
\max_{1\le n\le L}  \|u(t_n)-u^{n}\|_{L^2}
  \le C\tau^{\frac{4\gamma}{4+\gamma}};
\end{equation} 
\item[(2)] If $\gamma\in [1,2]$,  then 
\begin{equation}\label{error-tau-2}
\max_{1\le n\le L}  \|u(t_n)-u^{n}\|_{L^2}
  \le C\tau^{\frac25(1+\gamma)-},
\end{equation} 
\end{itemize}
where the constants $\tau_0$ and $C$ depend only on $T$ and $\|u^0\|_{H^\gamma}$. 
\end{theorem}

The following are some remarks for the main theorem. 
\begin{remark}\label{Remark}
{\upshape
(1) We have the uniform boundedness of the mass for the numerical solution by \eqref{NuSo-NLS-1}:
$$
\big\|u^n\big\|_{L^2_x}\le \big\|u^0\big\|_{L^2_x},  \quad n=0,1,\ldots, L - 1.
$$
and the conversed mass for the numerical solution by \eqref{NuSo-NLS-2}:
$$
\big\|u^n\big\|_{L^2_x}= \big\|u^0\big\|_{L^2_x},  \quad n=0,1,\ldots, L - 1.
$$
The latter one can not be  guaranteed by the filtered splitting method, see \cite{Ignat-2011,ORS-2022}. 

(2) If $\gamma\in (0,1)$,  
then the convergence rate obtained in Theorem \ref{the:main}  is  $\tau^{\frac{4\gamma}{4+\gamma}}$. 
This is better than $\tau^\frac \tau 2$, which is the expected convergence rate for the standard  (filtered) Lie or Strang  splitting method. In particular, for $H^1$-data, the convergence rate   is  $\tau^{\frac{4}5-}$. 

Moreover, as shown in \cite{ORS-2022}, there is a ``stability issue'' in the low regularity case: $0<\gamma\le \frac12$. In \cite{ORS-2022}, the authors overcome the difficulty by utilizing the discrete Bourgain spaces. 
In the present paper, we overcome it by the perturbation argument introduced in \cite{LiWu-2022}. Moreover, we  develop some new iteration idea to close the estimates in any long time.

(3) For the first-order convergence, the requirement of the regularity for the new scheme \eqref{NuSo-NLS-1} is $u_0\in H^{\frac32+}(\T)$, which is lower than that of the standard Strang splitting method: $u_0\in H^2(\T)$. 

While if $\gamma=2$,   Theorem \ref{the:main} yields the  convergence estimate: 
\begin{equation}\label{error-N-1}
\max_{1\le n\le L}  \|u(t_n,\cdot)-u^{n}\|_{L^2}
  \le C\tau^{\frac65-},
\end{equation} 
which is better than the expected first-order convergence for the standard Strang splitting method. 
}
\end{remark}

%
%
%

\subsection{Further discussions}

We believe that our argument also applicable for the Cauchy problem in whole space  and the other nonlinearities.  
Following, we give some claims for the sharp regularity condition. 

\begin{itemize}

\item{\bf Claim 1: $\frac32$-derivatives loss for first-order convergence}

One may further consider more general convergence estimate.  Then we claim that if $u^0\in H^{\gamma+\frac32},\gamma\ge 0$, then the numerical solution given by \eqref{NuSo-NLS-1} or \eqref{NuSo-NLS-2} has the following first-order convergent bound:
\begin{equation*} 
\max_{1\le n\le L}  \|u(t_n)-u^{n}\|_{H^\gamma}
  \le C\tau. 
\end{equation*} 
That is, it has the first-order convergence in $H^{\gamma}(\T)$ is $u^0\in H^{\gamma+\frac32}(\T)$ for $\gamma\ge 0$ (losing $\frac32$ derivatives).

\item{\bf Claim 2: Schemes for the general power-type nonlinearity}

One may further consider the numerical solution of the following nonlinear Schr\"odinger  equation: 
\begin{equation*}
 \left\{\begin{aligned}
& i\partial_tu(t,x)
=\partial_x^2 u(t,x)+\lambda|u(t,x)|^{2p}u(t,x)
 &&\mbox{for}\,\,\, x\in\T\,\,\,\mbox{and}\,\,\, t\in(0,T] , \\
 &u(0,x)=u^0(x) &&\mbox{for}\,\,\,  x\in\T,
 \end{aligned}\right.
\end{equation*}
with $p\in \Z^+$. 

We define  the modified time-splitting method: Let $L=\frac T\tau$,  for $n=0,1\ldots, L - 1$, 
\begin{align*}
u^{n+1}=\fe^{i\frac\tau2\partial_x^2}\widetilde{\mathcal N}_{\tau,p}
\left[\fe^{i\frac\tau2\partial_x^2}\big(\Pi_\tau +\fe^{-2\pi i p \tau \Pi_0(|u^n|^{2p}) }\Pi^\tau \big)u^n\right],
\end{align*}
where  the flow $\widetilde{\mathcal N}_{t,p}$ and $\Pi_0$ are  defined by  
$$
\widetilde{\mathcal N}_{t,p}(\phi)=\fe^{-i\lambda t|\Pi_\tau\phi|^{2p}}\phi,\quad 
\Pi_0(f)=\frac1{2\pi} \int_\T f(x)\,dx.
$$
Then we claim that such new time-splitting method presents the first-order convergence in $H^{\gamma}(\T)$ is $u^0\in H^{\gamma+\frac32}(\T)$ for $\gamma\ge 0$ (losing $\frac32$ derivatives).

\item{\bf Claim 3: Schemes for high dimensional cases}

One may further consider the numerical solution of the cubic nonlinear Schr\"odinger  equation  in the cases of dimensions $d\ge 2$: 
\begin{equation*}
 \left\{\begin{aligned}
& i\partial_tu(t,x)
=\Delta u(t,x)+\lambda|u(t,x)|^2u(t,x)
 &&\mbox{for}\,\,\, x\in\T^d\,\,\,\mbox{and}\,\,\, t\in(0,T] , \\
 &u(0,x)=u^0(x) &&\mbox{for}\,\,\,  x\in\T^d.
 \end{aligned}\right.
\end{equation*}

Denote $\fe^{i\tau\Delta}$ to be the linear flow, and the modified time-splitting method: Let $L=\frac T\tau$,  for $n=0,1\ldots, L - 1$, 
\begin{align*}
u^{n+1}=\Pi_{>N}\big[\fe^{i\tau\Delta-4\pi i\lambda M_0\tau}u^n\big]+\Pi_{\le N}\fe^{i\frac\tau2\Delta}\mathcal N_\tau\big[\fe^{i\frac\tau2\Delta}\Pi_{\le N}u^n\big].
\end{align*}
Then we claim that by suitably choosing the parameter $N$,  this new time-splitting method presents a higher-order convergence rate than the standard Strang splitting or Lie splitting method under the same regularity condition. 

%
%
%

 \end{itemize}

The proofs of the claims seem be more complicated, however, these will be an interesting analysis works for the further study. 
 
\subsection{Organization}
The rest of this article is organized as follows. Some notations and technical tools are presented in section \ref{section:tool}. In particular, a crucial trilinear estimate is given in this section. Construction of the numerical method is given in Section \ref{section:derivation}. In Section \ref{sec:Continuation}, we construct a continuous solution $\mathcal U_{\le N}(t),\mathcal U_{> N}(t)$ such that $\mathcal U_{\le N}(t_n)=\Pi_{\le N} u^n$ and $\mathcal U_{> N}(t_n)=\Pi_{> N} u^n$ respectively for $n=0,1,\cdots, L-1$. Then  framework of the {\it perturbation argument} is presented and the working spaces are given in this section. 
The consistency estimates for the local and global errors are established in Section \ref{section:errors}.
In Section \ref{Convergence}, we give the convergence estimates when the lifetime is suitably short. The method is close the working space by the bootstrap argument.  In Section \ref{proof-of-Theorem}, we extend the estimates in short time to any long time by the iteration. 

\section{Notation and technical tools}
\label{section:tool}

In this section we introduce the basic notation and technical lemmas to be used in construction of the numerical method and the analysis of the error. 

\subsection{Notation}\label{subsec1}
We denote by $A\lesssim B$ or $B\gtrsim A$ the statement 
$A\leq CB$ for some constant $C>0$.  
The value of $C$ may depend on $T$ and $\|u^0\|_{H^\gamma}$, and may be different at different occurrences, but is always independent of $\tau$, $N$ and $n$. 
The notation $A\sim B$ means that $A\lesssim B\lesssim A$. 
We denote by $O(Y)$ any quantity $X$ such that $X\lesssim Y$.  Denote $a_+=\max\{a,0\}$, and $a\pm=a\pm\epsilon$ for arbitrary small $\epsilon>0$. 

The inner product and norm on $L^2(\T)$ are denoted by  
$$
( f,g ) = 
\int_\T f(x) \overline{g(x)}\,dx
\quad\mbox{and}\quad
\|f\|_{L^2}= \sqrt{( f,f )} ,\,\,\,\mbox{respectively}. 
$$
The norm on the Sobolev space $H^s(\T)$, $s\in\R$, is denoted by 
$$
\big\|f\big\|_{H^s}^2
=2\pi\sum_{k\in \Z}(1+ k ^2)^s |\hat{f}_k|^2 . 
$$
For a function $f:[0,T]\times\T\rightarrow \C$ we denote by $\|f\|_{L^p_tH^\gamma_x(I)}$ its space-time Sobolev norm, defined by 
$$
\|f\|_{L^p_tH^\gamma_x(I)}
= 
\left\{\begin{aligned}
&\bigg(\int_I \|f(t)\|_{H^s}^p d t\bigg)^{\frac{1}{p}} && \mbox{for}\,\,\, p\in[1,\infty) ,\\[5pt] 
&{\rm ess\!}\sup \limits_{t\in I\,\,\,}\! \|f(t)\|_{H^s} && \mbox{for}\,\,\, p=\infty. 
\end{aligned}\right. 
$$
Similarly, since we always restrict the variable on $x\in\T$, we denote 
$$
\|f\|_{L^r_x}=\|f\|_{L^r_x(\T)};\quad 
\|f\|_{L^q_tL^r_x(I)}=\left\|\big\|f\big\|_{L^r_x(\T)}\right\|_{L^q_t(I)}.
$$

Let a smooth cut-off function $\eta\in C_0^\infty(\R)$ be real-valued and even,  such that $\eta(t)=1$ if $|t| \le1$ and $\eta(t)=0$ if $|t|>2$. 


The Fourier coefficients of a function $f$ on $\T$ are denoted by $\mathcal{F}_k[f]$ or simply $\hat{f}_k$, defined by 
$$
\hat{f}_k = \frac{1}{2\pi}\int_{\T}  \fe^{- i   kx }f( x )\,d x \quad\mbox{for}\,\,\, k\in\bZ .
$$
The Fourier inversion formula is given by 
$$
f( x )=\sum_{k\in \Z} \hat{f}_k \fe^{ i  kx } .
$$
The Fourier coefficients are known to have the following properties:
\begin{align*}
\begin{aligned}
\|f\|_{L^2}^2
 & =2\pi \sum\limits_{k\in \Z}\big|\hat{f_k}\big|^2 && \mbox{(Plancherel identity)}; \\
\mathcal{F}_k[fg]  &=\sum\limits_{k_1\in\Z}
  \hat{f}_{k - k _1}\hat{g}_{k _1}  && \mbox{(Convolution)}.
  \end{aligned}
\end{align*}

For abbreviation, we denote 
$$\l k \r= (1+k^2)^{\frac{1}{2}} \quad\mbox{and}\quad J^s=  \l i^{-1} \partial_x \r^s ,$$
which imply that 
$$
\big\|f\big\|_{H^s}^2=\big\|J^sf\big\|_{L^2}^2 \quad\mbox{and}\quad
\widehat{(J^s f) } _k = \l k \r^s \hat f_k . 
$$
Moreover, we denote by $\partial_x^{-1}:H^s(\T)\rightarrow H^{s+1}(\T)$, $s\in\R$, the operator such that 
\begin{equation}\label{def:px-1}
\mathcal{F}_k[\partial_x^{-1}f]
=\Bigg\{ \aligned
    &(i k )^{-1}\hat{f}_k,\quad &\mbox{when }  k \ne 0,\\
    &0,\quad &\mbox{when }  k = 0.
   \endaligned
\end{equation}

%
Moreover, we denote 
\begin{align*}
\Pi_N f =\sum\limits_{k\in\Z, \frac18 N\le |k|\le 8N}\hat{f}_k   \fe^{ i  kx } .
\end{align*}


\subsection{Some technical lemmas}\label{subsec3}
We will use the following version of the Kato--Ponce inequalities, which was originally proved 
 in \cite{Kato-Ponce}
and subsequently improved to cover the endpoint case in 
\cite{BoLi-KatoPonce, Li-KatoPonce}.

\begin{lemma}[The Kato--Ponce inequalities] \label{lem:kato-Ponce} 
Let $f,g$ be the Schwartz functions. Then for $s>0$, $1<p< \infty$, and $1<p_1,p_2,p_3, p_4 < \infty$ with $\frac1p=\frac1{p_1}+\frac1{p_2}$, $\frac1p=\frac1{p_3}+\frac1{p_4}$, the following inequality holds:
\begin{align*}
  \big\|J^s(fg)\big\|_{L^p}\le C\Big( \|J^sf\|_{L^{p_1}}\|g\|_{L^{p_2}}+ \|J^sg\|_{L^{p_3}}\|f\|_{L^{p_4}}\Big),
\end{align*}
where the constant $C>0$ depends on $s,p,p_1,\cdots,p_4$. 
\end{lemma}

We also need the following multiplier estimate in the periodic setting, see \cite{Lpmulitiplier-mathoverflow} for its proof.
\begin{lemma}[The multiplier estimates] \label{lem:multiplier} 
$\,$
Let $p\in (1,+\infty)$, and the operator $T$ be defined by 
$$
\widehat{(Tf)}_k=m(k) \hat f_k.
$$
Then for any $m$ satisfies that 
\begin{align}\label{cond-m}
\sum\limits_{k}|m(k+1)-m(k)\big|\le A,\quad \mbox{and}\quad \lim\limits_{k\to -\infty} m(k)=0,
\end{align}
and any $f\in L^p(\T)$, 
\begin{align*}
\|Tf\|_{L^p} \lesssim  A\|f\|_{L^p}.
\end{align*}
\end{lemma}

Next, we present the Strichartz estimates  in the periodic setting. 
\begin{lemma} \label{lem:Stri}
For any $\epsilon>0$, there exists $C=C(\epsilon)>0$ such that the following inequalities hold:
\begin{itemize}
\item[(1)]  Let $f\in L^2(\T)$, then
$$
\big\|\fe^{it\partial_x^2}f\big\|_{L^\infty_tL^2_x([0,1])}+
\big\|\fe^{it\partial_x^2}f\big\|_{L^4_{tx}([0,1])}+
\big\|J^{-\epsilon}\fe^{it\partial_x^2}f\big\|_{L^6_{tx}([0,1])}
\le C\|f\|_{L^2}.
$$
\item[(2)] Let $F\in L^\frac43_{tx}([0,1]\times\T)$, then
$$
\Big\|\eta(t)\int_0^t \fe^{i(t-s)\partial_x^2} F(s)\,ds\Big\|_{L^\infty_tL^2_x\cap L^4_{tx} ([0,1])}
+\Big\|\eta(t)\int_0^t \fe^{i(t-s)\partial_x^2} J^{-\epsilon}F(s)\,ds\Big\|_{L^6_{tx} ([0,1])}
\le C\big\|F\big\|_{L^{\frac43}_{tx}([0,1])}.
$$
\end{itemize}
\end{lemma}


\subsection{A trilinear estimate}\label{subsec4}

The following trilinear estimate plays a crucial role in the proof of the main result. Let $M_j, j=0,\cdots, 3$ be the multipliers satisfy that 
$$
\widehat{T_jf}_k=m_j(k)\hat f_k,
$$
with $m_j$ verifying \eqref{cond-m}.
Let $I=[t_0,t_1]\subset \R$, we define the operator $T$ as 
\begin{align}\label{def-T}
Tv\triangleq& i\int_{t_0}^{t_1} \fe^{-is\partial_x^2} T_0\Big[\fe^{-is\partial_x^2}T_1\bar v(s)\cdot \fe^{is\partial_x^2}T_2 v(s)\cdot \fe^{is\partial_x^2}T_3 v(s)\Big]\,ds\notag\\
&-i\int_{t_0}^{t_1} T_0T_3v(s)\cdot \int_\T T_1\bar v(s)\cdot T_2 v(s)\,dx \,ds
-i\int_{t_0}^{t_1} T_0T_2v(s)\cdot \int_\T T_1\bar v(s)\cdot T_3 v(s)\,dx \,ds.
\end{align}
The nonlinear operator $T$ looks very complicated, however, it well reflects the structure of the nonlinear equation \eqref{model}. 
Let $\gamma\in \R $, denote
\begin{align}\label{def:alpha}
\alpha(\gamma)\triangleq 1+\gamma-\frac{(1-2\gamma)_+^2}{1+(1-2\gamma)_+}+.
\end{align}
Then we have the following lemma. 
\begin{lemma}\label{lem:tri-est}
Denote the norm 
$$
\|v\|_{X^\gamma([t_0,t_1])}\triangleq \|v\|_{L^\infty_tH^{\gamma}_x([t_0,t_1])}+\|\partial_tv\|_{L^1_tH^{\gamma}_x([t_0,t_1])},
$$
and  let $\gamma\in \R $ such that 
$
\alpha(\gamma)\ge 0,
$
then it holds that 
\begin{align*}
\big\|Tv\big\|_{L^2_x}\lesssim  &
\sum\limits_{j} 2^{-3\gamma j} \left\|\Pi_{2^j} \fe^{is\partial_x^2}J^\gamma v\right\|_{L^3_tL^6_x([t_0,t_1])}^3\\
&\quad 
+ \sum\limits_{|j_1|\ge |j_2|\ge|j_3|} 2^{-\alpha j_1}\big\|\Pi_{2^{j_1}}v\big\|_{X^\gamma([t_0,t_1])}\big\|\Pi_{2^{j_2}}v\big\|_{X^\gamma([t_0,t_1])}\big\|\Pi_{2^{j_3}}v\big\|_{X^\gamma([t_0,t_1])}.
\end{align*}
\end{lemma}
\begin{proof} In this proof, we default the integral region to be $[t_0,t_1]\times \T$. 
Note that 
\begin{align*}
\widehat{(Tv)}_k=&i \int_{t_0}^{t_1}\sum\limits_{\substack{k_1+k_2+k_3=k\\ k_1+k_2\ne0,k_1+k_3\ne0}}\fe^{is\phi} \mathcal M(k,k_1,k_2,k_3) \>\hat{ \bar v}_{ k _1}(s)\hat v_{ k _2}(s)\hat v_{ k _3}(s)\,ds\\
&\quad
+i\int_{t_0}^{t_1}\mathcal M(k,-k,k,k) \>\hat{ \bar v}_{-k}(s)\hat v_{ k}(s)\hat v_{ k}(s)\,ds\\
\triangleq & \widehat{(T_1v)}_k+\widehat{(T_2v)}_k,
\end{align*}
where we denote 
$$
\mathcal M(k,k_1,k_2,k_3)=m_0(k)m_1(k_1)m_2(k_2)m_3(k_3),
$$
and 
\begin{align*}
\phi 
= k ^2+ k _1^2- k _2^2- k _3^2.
\end{align*}

\noindent {\bf Estimate on $T_1v$}. 
We denote $k_m: |k_m|\triangleq \max\{|k_1|,|k_2|,|k_3||\}$, and the sets 
\begin{align*}
\Gamma_k^{\phi\ne 0}\triangleq &\{(k_1,k_2,k_3):k_1+k_2+k_3=k,k_1+k_2\ne0,k_1+k_3\ne0\};\\ 
\Gamma_{k}^1\triangleq &\{(k_1,k_2,k_3)\in \Gamma_k^{\phi\ne 0}:|k_1|\sim |k_2|\sim|k_3|\sim |k| \};\\
\Gamma_{k}^2\triangleq & \Gamma_k^{\phi\ne 0}\setminus \Gamma_{k}^1.
\end{align*}
Then 
\begin{align*}
\widehat{(T_1v)}_k&=\sum\limits_{j=1}^2 i \int_{t_0}^{t_1}\sum\limits_{\Gamma_{k}^j}\fe^{is\phi} \mathcal M(k,k_1,k_2,k_3) \>\hat{ \bar v}_{ k _1}(s)\hat v_{ k _2}(s)\hat v_{ k _3}(s)\,ds\\
&\triangleq\sum\limits_{j=1}^2\widehat{(T_{1,j}v)}_k.
\end{align*}

$ \bullet $ Estimate on $T_{1,1}v$. By the dyadic decomposition, we write 
\begin{align*}
T_{1,1}v=&i\sum\limits_{j}\int_{t_0}^{t_1} \fe^{-is\partial_x^2} \Pi_{2^j}T_0\Big[\fe^{-is\partial_x^2}\Pi_{2^j}T_1\bar v(s)\cdot \fe^{is\partial_x^2}\Pi_{2^j}T_2 v(s)\cdot \fe^{is\partial_x^2}\Pi_{2^j}T_3 v(s)\Big]\,ds\\
&\quad +\mbox{other similar terms}.
\end{align*}
Here ``other similar terms'' stand for the terms dropping the restriction $k_1+k_2\ne0,k_1+k_3\ne0$.
Then by Lemma \ref{lem:multiplier}, we have that 
\begin{align}
\big\|T_{1,1}v\big\|_{L^2}
\lesssim \sum\limits_{j} 2^{-3\gamma j} \left\|\Pi_{2^j} \fe^{is\partial_x^2}J^\gamma v\right\|_{L^3_tL^6_x}^3.
\end{align}

$ \bullet $ Estimate on $T_{1,2}v$. Note that $\phi\ne 0$, using 
$$
\fe^{{is\phi}}=\frac1{i\phi}\partial_s\big(\fe^{{is\phi}}\big),
$$
and integration by parts, we get that 
\begin{subequations}\label{T-v11}
\begin{align}
\widehat{(T_{1,2}v)}_k&= i\sum\limits_{\Gamma_{k}^2}\frac{1}{i\phi}\fe^{is\phi} \mathcal M(k,k_1,k_2,k_3) \>\hat{ \bar v}_{ k _1}(s)\hat v_{ k _2}(s)\hat v_{ k _3}(s)\Big|_{t_0}^{t_1}\label{T-v11-BT}\\
&\qquad -i \int_{t_0}^{t_1}\sum\limits_{\Gamma_{k}^2}\frac{1}{i\phi}\fe^{is\phi} \mathcal M(k,k_1,k_2,k_3) \>\partial_s\Big[\hat{ \bar v}_{ k _1}(s)\hat v_{ k _2}(s)\hat v_{ k _3}(s)\Big]\,ds.\label{T-v11-5T}
\end{align}
\end{subequations}

For \eqref{T-v11-BT}, we first note that in $\Gamma_k^2$, $|k_1+k_2|\gtrsim |k_m|$ or $|k_1+k_3|\gtrsim |k_m|$. 
Without loss of generality, we may restrict $|k_1+k_2|\gtrsim |k_m|$ in $\Gamma_k^2$, since the other cases can be treated the same. Furthermore, we consider the cases in the following two subsets separately:
\begin{align*}
\Gamma_{k}^{2,1}\triangleq &\{(k_1,k_2,k_3)\in \Gamma_k^{\phi\ne 0}:|k_1+k_2|\gtrsim |k_m| \mbox{ or } |k_1+k_3|\gtrsim |k_m|^l \};\\
\Gamma_{k}^{2,2}\triangleq &\{(k_1,k_2,k_3)\in \Gamma_k^{\phi\ne 0}:|k_1+k_2|\gtrsim  |k_m| \mbox{ and } |k_1+k_3|\ll |k_m|^l \},
\end{align*}
and thus roughly
\begin{align*}
\eqref{T-v11-BT}&= \sum\limits_{j=1}^2 i\sum\limits_{\Gamma_{k}^{2,j}}\frac{1}{i\phi}\fe^{is\phi} \mathcal M(k,k_1,k_2,k_3) \>\hat{ \bar v}_{ k _1}(s)\hat v_{ k _2}(s)\hat v_{ k _3}(s)\Big|_{t_0}^{t_1}\\
&\triangleq (\ref{T-v11-BT}1)+ (\ref{T-v11-BT}2).
\end{align*}

For (\ref{T-v11-BT}1), since $|\mathcal M|\lesssim 1 $, we have that 
\begin{align*}
|(\ref{T-v11-BT}1)|&\lesssim \max\limits_{s\in \{t_0,t_1\}}\sum\limits_{k_1+k_2+k_3=k}|k_m|^{-1-l}\>\big|\hat{ \bar v}_{ k _1}(s)\big|\big|\hat v_{ k _2}(s)\big|\big|\hat v_{ k _3}(s)\big|.
\end{align*}
By the dyadic decomposition and Cauchy-Schwarz inequality, we have that for any $\gamma\in\R$ with 
$$
\alpha_1\triangleq 1+l+\gamma-(1-2\gamma)_+\ge 0 ,
$$
it follows that 
\begin{align*}
\|(\ref{T-v11-BT}1)\|_{l^2_k}&\lesssim \sum\limits_{|j_1|\ge |j_2|\ge|j_3|} 2^{-\alpha_1 j_1} \big\|\Pi_{2^{j_1}}v\big\|_{L^\infty_tH^{\gamma}_x}\big\|\Pi_{2^{j_2}}v\big\|_{{L^\infty_tH^{\gamma}_x}}\big\|\Pi_{2^{j_3}}v\big\|_{{L^\infty_tH^{\gamma}_x}}.
\end{align*}
 
For (\ref{T-v11-BT}2),  we have that 
\begin{align*}
|(\ref{T-v11-BT}2)|&\lesssim \max\limits_{s\in \{t_0,t_1\}}\sum\limits_{k_3}\sum\limits_{|\tilde{k_1}|\ll |k_m|^l}|k_m|^{-1}\>\big|\hat{ \bar v}_{ \tilde{k_1}-k_3}(s)\big|\big|\hat v_{ k- \tilde{k_1}}(s)\big|\big|\hat v_{ k _3}(s)\big|.
\end{align*}
Again, by the dyadic decomposition and Cauchy-Schwarz inequality, we have that 
 for any $\gamma\in\R$ with 
$$
\alpha_2\triangleq 1+\gamma-(1-2\gamma)_+l\ge 0 ,
$$
it follows that 
\begin{align*}
\|(\ref{T-v11-BT}2)\|_{l^2_k}&\lesssim \sum\limits_{|j_1|\ge |j_2|\ge|j_3|} 2^{-\alpha_2 j_1}\big\|\Pi_{2^{j_1}}v\big\|_{L^\infty_tH^{\gamma}_x}\big\|\Pi_{2^{j_2}}v\big\|_{{L^\infty_tH^{\gamma}_x}}\big\|\Pi_{2^{j_3}}v\big\|_{L^\infty_tH^{\gamma}_x}.
\end{align*}

In particular, choosing 
$$
l=\frac{(1-2\gamma)_+}{1+(1-2\gamma)_+},  
$$
such that $\alpha_1=\alpha_2=\alpha$, 
and combining with the two estimates on (\ref{T-v11-BT}), we obtain that 
\begin{align*}
\|(\ref{T-v11-BT})\|_{l^2_k}&\lesssim \sum\limits_{|j_1|\ge |j_2|\ge|j_3|} 2^{-\alpha j_1}\big\|\Pi_{2^{j_1}}v\big\|_{L^\infty_tH^{\gamma}_x}\big\|\Pi_{2^{j_2}}v\big\|_{{L^\infty_tH^{\gamma}_x}}\big\|\Pi_{2^{j_3}}v\big\|_{{L^\infty_tH^{\gamma}_x}}.
\end{align*}

Arguing similarly as (\ref{T-v11-BT}), we also have that 
\begin{align*}
\|(\ref{T-v11-5T})\|_{l^2_k}
&\lesssim \sum\limits_{|j_1|\ge |j_2|\ge|j_3|} 2^{-\alpha j_1}\Big[\big\|\Pi_{2^{j_1}}\partial_tv\big\|_{L^1_tH^{\gamma}_x}\big\|\Pi_{2^{j_2}}v\big\|_{L^\infty_tH^{\gamma}_x}\big\|\Pi_{2^{j_3}}v\big\|_{L^\infty_tH^{\gamma}_x}\\
&\qquad\qquad +\big\|\Pi_{2^{j_1}}v\big\|_{L^\infty_tH^{\gamma}_x}\big\|\Pi_{2^{j_2}}\partial_tv\big\|_{L^1_tH^{\gamma}_x}\big\|\Pi_{2^{j_3}}v\big\|_{{L^\infty_tH^{\gamma}_x}}
\\
&\qquad\qquad+\big\|\Pi_{2^{j_1}}v\big\|_{L^\infty_tH^{\gamma}_x}\big\|\Pi_{2^{j_2}}v\big\|_{L^\infty_tH^{\gamma}_x}\big\|\Pi_{2^{j_3}}\partial_tv\big\|_{L^1_tH^{\gamma}_x}\Big].
\end{align*}

Together with the two estimates on \eqref{T-v11}, it gives that 
\begin{align*}
\big\|T_{1,2}v\big\|_{L^2}
\lesssim 
 \sum\limits_{|j_1|\ge |j_2|\ge|j_3|} 2^{-\alpha j_1}\big\|\Pi_{2^{j_1}}v\big\|_{X^\gamma([t_0,t_1])}\big\|\Pi_{2^{j_2}}v\big\|_{X^\gamma([t_0,t_1])}\big\|\Pi_{2^{j_3}}v\big\|_{X^\gamma([t_0,t_1])}.
\end{align*}


\noindent {\bf Estimate on $T_2v$}. Treated similarly as $T_{1,1}v$, we also have that 
\begin{align*}
\big\|T_2v\big\|_{L^2}
\lesssim \sum\limits_{j} 2^{-3\gamma j} \left\|\Pi_{2^j} \fe^{is\partial_x^2}J^\gamma v\right\|_{L^3_tL^6_x}^3.
\end{align*}

Collecting with the estimates on $T_1v$ and  $T_2v$, we obtain the desired estimates. 
\end{proof}

%

\subsection{Local theory and the related estimates}
We recall the following well-known global result for \eqref{model}, see Bourgain \cite{Bo}.
\begin{lemma}\label{lem:global-theory}
Let $\gamma>0$, then the  problem  \eqref{model} is globally well-posed in $H^\gamma$. In particular, for any $T>0$, 
such that for some constant $C(T,\|u_0\|_{H^\gamma})> 0$,
\begin{align}\label{solution-est}
\big\|u\big\|_{L^\infty_tL^2_x([0,T^*]\times \T)}+\big\|J^\gamma u\big\|_{L^4_{tx}([0,T^*]\times \T)}
+\big\|J^{\gamma-} u\big\|_{L^6_{tx}([0,T^*]\times \T)}
\le C(T,\|u_0\|_{H^\gamma}).
\end{align}
\end{lemma}

Some consequences of  the lemmas above are in the following. The first is
\begin{lemma}\label{lem:est-vt}
Under the same assumption  as  Lemma \ref{lem:global-theory}, 
\begin{align*} 
\tau^\frac16\big\|J^{\gamma-}\Pi_{\le N}u(t_n)\big\|_{l^6_kL^6_{x}(\T)}
\le C(T,\|u_0\|_{H^\gamma})\big(1+N^2\tau).
\end{align*}
\end{lemma}
\begin{proof}
By Sobolev's inequality, we have that 
\begin{align*} 
\tau^{\frac16}\big\|J^{\gamma-}\Pi_{\le N}u(t_n)\big\|_{L^6_{x}}
=&\big\|J^{\gamma-}\Pi_{\le N}u(t_n)\big\|_{L^6_{tx}([t_n,t_{n+1}])}\\
\le &\big\|J^{\gamma-}u(t)\big\|_{L^6_{tx}([t_n,t_{n+1}])}
+\big\|J^{\gamma-}\Pi_{\le N}\big(u(t)-u(t_n)\big)\big\|_{L^6_{tx}([t_n,t_{n+1}])}.
\end{align*}
Note that by \eqref{model}, \eqref{mass}, Lemma \ref{lem:kato-Ponce}, and Bernstein's inequality, 
\begin{align*}
\big\|J^{\gamma-}\Pi_{\le N}&\big(u(t)-u(t_n)\big)\big\|_{L^6_{tx}([t_n,t_{n+1}])}\\
\lesssim &
 \big\|J^{\gamma-}\Pi_{\le N}\int_{t_n}^t\partial_s u(s)\,ds\big\|_{L^6_{tx}([t_n,t_{n+1}])}\\
 \lesssim &
\tau \big\|\Pi_{\le N}J^{\gamma-}\Delta u(t)\big\|_{L^6_{tx}([t_n,t_{n+1}])}
+\tau \big\|J^{\gamma-} \Pi_{\le N}\big(|u(t)|^2u(t)\big)\big\|_{L^6_{tx}([t_n,t_{n+1}])}
\\
\lesssim &
\tau N^2 \big\|J^{\gamma-} u(t)\big\|_{L^6_{tx}([t_n,t_{n+1}])}
 +\tau N \big\|J^{\gamma-}u\big\|_{L^6_{tx}([t_n,t_{n+1}])}\|u\|_{L^\infty_{t}L^2_x([t_n,t_{n+1}])}^2\\
\lesssim &
\tau N^2\big\|J^{\gamma-} u\big\|_{L^6_{tx}([t_n,t_{n+1}])}.
\end{align*}
Therefore, 
\begin{align*} 
\tau^{\frac16}\big\|J^{\gamma-}\Pi_{\le N}u(t_n)\big\|_{L^6_{x}(\T)}
\le C\big(1+\tau N^2\big)\big\|J^{\gamma-} u\big\|_{L^6_{tx}([t_n,t_{n+1}])}.
\end{align*}
By \eqref{solution-est}, it gives the desired estimate. 
\end{proof}

Moreover, we have 
\begin{lemma}\label{lem:est-vt-2}
Let $\gamma>0$, and $u$ is a solution obtained in Lemma \ref{lem:global-theory} and $v(t)=\fe^{-it\partial_x^2}u(t)$, then 
\begin{align*} 
\sum\limits_{n=1}^{L-1}\big\|J^{\gamma-}\big(v(t_n+s)-v(t_n)\big)\big\|_{L^\infty_sL^2_{x}([0,\tau]\times \T)}
\le C(T,\|u_0\|_{H^\gamma}).
\end{align*}
\end{lemma}
\begin{proof}
Note that 
\begin{align*}
\partial_{t}  v(t)
=&-i \lambda \fe^{-it\partial_x^2}
    \big[|\fe^{it\partial_x^2}v(t)|^2\,\fe^{it\partial_x^2}v(t)\big] 
\end{align*}
(see \eqref{pt-v} below). Then by Lemma \ref{lem:kato-Ponce}, it follows that 
\begin{align*}
\big\|J^{\gamma-}\partial_t v\big\|_{L^2_{x}(\T)}
\lesssim 
\big\|J^{\gamma-}\big(|u(s)|^2u(s)\big)\big\|_{L^2_{x}(\T)}
\lesssim 
\big\|J^{\gamma-}u\big\|_{L^6_{x}(\T)}^3.
\end{align*}
Thus 
\begin{align*} 
\sum\limits_{n=1}^{L-1}\big\|J^{\gamma-}\big(v(t_n+s)-v(t_n)\big)\big\|_{L^\infty_sL^2_{x}([0,\tau]\times \T)}
\le  &\sum\limits_{n=1}^{L-1}\int_0^\tau\big\|J^{\gamma-}\partial_t v(t_n+t)\big\|_{L^2_{x}(\T)}\,dt\\
\le & \big\|J^{\gamma-}u\big\|_{L^3_tL^6_{x}([0,T]\times \T)}^3.
\end{align*}
Then the desired estimate is followed from \eqref{solution-est}. 
\end{proof}


\section{Construction of the numerical method}\label{section:derivation}

\subsection{Duhamel's formula}
By Duhamel's formula: 
\begin{align}\label{Duhamel-u}
u(t_{n}+t)=\fe^{it\partial_x^2}u(t_n)
 -  i \lambda \int_{t_n}^t \fe^{i\left(t-s\right)\partial_x^2}
   \big(|u(s)|^2u(s)\big) \,ds ,
\end{align} 
as well as the mass conservation \eqref{mass}. 
Let $v(t):=\fe^{-it\partial_x^2}u(t)$ be the twisted variable. Then $v\in C([0,T];H^\gamma(\T))$ for any $\gamma\ge 0$, and satisfies $\|v\|_{C([0,T];H^\gamma(\T))}=\|u\|_{C([0,T];H^\gamma(\T))}$. In particular,  the following mass conservation law holds: 
\begin{align}
\frac1{2\pi}\int_\T |v(t,x)|^2\,d x = \frac1{2\pi}\int_\T |u(t,x)|^2\,d x = M_0 \quad\mbox{for}\,\,\, t>0 .  \label{mass-v}
\end{align}

Applying the operator $\fe^{-it_{n+1} \partial_x^2}$ to the identity \eqref{Duhamel-u}, we obtain 
\begin{align}\label{solution-t}
v(t_{n}+t)=v(t_n)-i \lambda \int_{t_n}^t \fe^{-is\partial_x^2}
    \big[|\fe^{is\partial_x^2}v(s)|^2\,\fe^{is\partial_x^2}v(s)\big]\,ds.
\end{align}
This gives that 
\begin{align}
\partial_{t} v(t)
=&-i \lambda\fe^{-it\partial_x^2}
    \big[|\fe^{it\partial_x^2}v(t)|^2\,\fe^{it\partial_x^2}v(t)\big].\label{pt-v}
\end{align}
In particular, \eqref{solution-t} gives 
\begin{align}\label{solution}
v(t_{n+1})=v(t_n)-i \lambda \int_{t_n}^{t_{n+1}} \fe^{-is\partial_x^2}
    \big[|\fe^{is\partial_x^2}v(s)|^2\,\fe^{is\partial_x^2}v(s)\big]\,ds.
\end{align}
The Fourier coefficients of both sides of \eqref{solution} should be equal, i.e.,  
\begin{align}\label{solution-F}
\hat v_{ k }(t_{n+1})
  =\hat v_{ k }(t_n)
     -i \lambda \int_{t_n}^{t_{n+1}} \sum\limits_{ k _1+ k _2+ k _3 = k }\fe^{is\phi}
        \>\hat{ \bar v}_{ k _1}(s)\hat v_{ k _2}(s)\hat v_{ k _3}(s)\,ds , 
\end{align}
with the phase function  
\begin{align*}
\phi 
=\phi( k, k _1, k _2, k _3) 
= k ^2+ k _1^2- k _2^2- k _3^2
=2(k_1+k_2)(k_1+k_3).
\end{align*}


\subsection{Approximation}

{\bf Step 1: Dropping part of the high-frequency terms.}
From \eqref{solution}, we write 
\begin{align}\label{solution-step1}
v(t_{n+1})=&v(t_n)-i \lambda \int_{t_n}^{t_{n+1}} \fe^{-is\partial_x^2}\Pi_{\le N}
    \big[|\fe^{is\partial_x^2}\Pi_{\le N}v(s)|^2\,\fe^{is\partial_x^2}\Pi_{\le N}v(s)\big]\,ds\notag\\
    &\quad -4\pi i\lambda M_0 \tau \Pi_{>N}v(t_n)+R_n^1(t_{n+1}),
\end{align}
where the remainder term $R^1_n$ is defined by 
\begin{align*}
R^1_n(s)
=& -i \lambda \int_{t_n}^s\fe^{-is\partial_x^2} \Pi_{>N}\Big[\fe^{-is\partial_x^2}\Pi_{\le N}\bar v(s)\cdot \fe^{is\partial_x^2} \Pi_{\le N}v(s)\cdot \fe^{is\partial_x^2} \Pi_{\le N}v(s)\Big]\,ds\\
& -i \lambda \int_{t_n}^s\fe^{-is\partial_x^2} \Big[\fe^{-is\partial_x^2}\Pi_{> N}\bar v(s)\cdot \fe^{is\partial_x^2} \Pi_{\le N}v(s)\cdot \fe^{is\partial_x^2}\Pi_{\le N} v(s)\Big]\,ds\\
& -i \lambda \int_{t_n}^s\fe^{-is\partial_x^2} \Big[\fe^{-is\partial_x^2}\bar v(s)\cdot \fe^{is\partial_x^2} \Pi_{> N}v(s)\cdot \fe^{is\partial_x^2}\Pi_{\le N} v(s)\Big]\,ds\\
&\quad 
+i \lambda \int_{t_n}^s \Pi_{\le N} v(s)\int_\T|\Pi_{> N}v(s)|^2\,dx \,ds
+i \lambda \int_{t_n}^s \Pi_{> N} v(s)\int_\T|\Pi_{\le N}v(s)|^2\,dx \,ds\\
& -i \lambda \int_{t_n}^s\fe^{-is\partial_x^2} \Big[\fe^{-is\partial_x^2}\bar v(s)\cdot \fe^{is\partial_x^2} v(s)\cdot \fe^{is\partial_x^2}\Pi_{> N} v(s)\Big]\,ds\\
&\quad 
+i \lambda \int_{t_n}^s  v(s)\int_\T|\Pi_{> N}v(s)|^2\,dx \,ds
+i \lambda \int_{t_n}^s \Pi_{> N} v(s)\int_\T|v(s)|^2\,dx \,ds\\
  & -2i \lambda \int_{t_n}^s \Pi_{\le N}  v(s)\int_\T|\Pi_{> N}v(s)|^2\,dx \,ds\\
  &+ 4\pi i \lambda M_0  \int_{t_n}^s \Pi_{>N}\>\big[v(s)-v(t_n)\big]\,ds.
\end{align*}

From \eqref{solution-step1}, we obtain that 
\begin{align}
\Pi_{\le N}v(t_{n+1})=&\Pi_{\le N}v(t_n)-i \lambda  \int_{t_n}^{t_{n+1}} \fe^{-is\partial_x^2}\Pi_{\le N}
    \big[|\fe^{is\partial_x^2}\Pi_{\le N}v(s)|^2\,\fe^{is\partial_x^2}\Pi_{\le N}v(s)\big]\,ds\notag\\
     &+\Pi_{\le N}R_n^1(t_{n+1});\label{solution-step1-low}\\
 \Pi_{> N}v(t_{n+1})=&\Pi_{> N}v(t_n) -4\pi i\lambda M_0 \tau \Pi_{>N}v(t_n)+\Pi_{> N} R_n^1(t_{n+1}).   \label{solution-step1-high}
\end{align}

{\bf Step 2: Dropping the high-order terms.}

For convenience, we denote 
\begin{align}
F_n(v;s)\triangleq&-i \lambda \int_{t_n}^s \fe^{-it\partial_x^2}\Pi_{\le N}
    \big[\big|
    \fe^{it\partial_x^2}\Pi_{\le N}v(t)\big|^2\fe^{it\partial_x^2}\Pi_{\le N}v(t)\big]\,dt.
\end{align}
Then replacing $\tau$ by $s$ in \eqref{solution-step1-low}, we have a  general formula:
\begin{align*} 
\Pi_{\le N}v(s)=\Pi_{\le N}v(t_n)+F_n(v;s)+\Pi_{\le N}R_n^1(s)\quad \mbox{ for any } s \in [t_n,t_{n+1}].
\end{align*}

Next, we approximate $F_n(v;t_{n+1})$.
Note that 
$$
\partial_s F_n(v;t_n)=-i \lambda  \fe^{-it_n\partial_x^2}
    \Pi_{\le N}\big[|\Pi_{\le N}u(t_n)|^2\Pi_{\le N}u(t_n)\big],
$$
 we can approximate $v(s)$ as 
\begin{align*}
 \Pi_{\le N}v(s)&\approx \Pi_{\le N}v(t_n)+(s-t_n)\partial_s F_n(v;t_n)\big)\notag\\
 &=\Pi_{\le N}v(t_n)-i\lambda (s-t_n)  \fe^{-it_n\partial_x^2}\Pi_{\le N}\big[|\Pi_{\le N}u(t_n)|^2\Pi_{\le N}u(t_n)\big].
\end{align*}
Denote $u=\fe^{it\partial_x^2}v$ and 
\begin{align}\label{def-G}
G_n(v)\triangleq&-i \lambda  \int_{t_n}^{t_{n+1}} \fe^{-is\partial_x^2}\Pi_{\le N}
    \Big[\Big|\fe^{is\partial_x^2}\Big(v-i\lambda (s-t_n) \fe^{-it_n\partial_x^2} \Pi_{\le N}\big[|u|^2u\big]\Big)\Big|^2\notag\\
   &\qquad \cdot\fe^{is\partial_x^2}\Big(v-i\lambda (s-t_n) \fe^{-it_n\partial_x^2} \Pi_{\le N}\big[|u|^2u\big]\Big)\Big]\,ds;\\
R_n^2\triangleq &  F_n(v;t_{n+1})- G_n\big(\Pi_{\le N}v(t_n)\big).\notag
\end{align}
Then 
\begin{align}\label{Fn-Gn-vn}
F_n(v;t_{n+1})
=&G_n\big(\Pi_{\le N}v(t_n)\big)+R_n^2.
\end{align}
Since $R_n^2$ is a high-order term, we will drop  $R_n^2$ in the definition of the numerical solution in the following. 
Then the approximation on  $F_n(v;t_{n+1})$ is now turning to the one on  $G_n\big(\Pi_{\le N}v(t_n)\big)$.

%

{\bf Step 3: Freezing the phase.}

In the following, we denote 
\begin{align}\label{vu}
v=\Pi_{\le N}v(t_n)\quad \mbox{and}\quad  u=\fe^{it_n\partial_x^2}v
\end{align}
for short.

Now we freeze the phase and  write 
\begin{align}
G_n(v)
=&-i \lambda  \int_{t_n}^{t_{n+1}} \fe^{-i(t_n+\frac\tau 2)\partial_x^2}\Pi_{\le N}
    \Big[\Big|\fe^{i\frac\tau 2\partial_x^2}\Big(u-i\lambda (s-t_n) \Pi_{\le N}\big[|u|^2u\big]\Big)\Big|^2\notag\\
   &\quad \cdot\fe^{i\frac\tau 2\partial_x^2}\Big(u-i\lambda (s-t_n) \Pi_{\le N}\big[|u|^2u\big]\Big)\Big]\,ds+R_n^3[u],\label{Fn-tn+1-3}
\end{align}
where $R_n^3[u]$ is defined by 
\begin{align}\label{Def-Rn3}
&\mathcal F_k(R_n^3[u])\triangleq  -i \lambda \int_0^\tau\sum\limits_{\Gamma_{k,N} }\fe^{it_nk^2}\>r_0(s)
        \Big(\widehat{\bar u}_{k_1}+i\lambda s\widehat{\big(|u|^2\bar u\big)}_{k_1}\Big)
        \notag\\
        &\quad \cdot \Big(\hat u_{k_2}-i\lambda s \widehat{\big(|u|^2u\big)}_{k_2}\Big)
        \Big(\hat u_{k_3}-i\lambda s \widehat{\big(|u|^2u\big)}_{k_3}\Big)\,ds,
\end{align}
here  we have denoted the sets 
\begin{align}\label{def-r0r1}
\Gamma_{k,N}\triangleq&\{(k_1,k_2,k_3):k_1+k_2+k_3=k, |k_1+k_2+k_3|\le N, |k_j|\le N,j=1,2,3\},
\end{align}
and 
$$
r_0(s) \triangleq \fe^{is\phi}-\fe^{\frac i2\tau\phi}.
$$

{\bf Step 4: Dropping the other high-order terms.}

From \eqref{Fn-tn+1-3}, we draw out the third-order terms and write 
\begin{align}
G_n(v)
=&-i \lambda  \int_{t_n}^{t_{n+1}} \fe^{-i(t_n+\frac\tau 2)\partial_x^2}\Pi_{\le N}
    \Big[\fe^{-i\frac\tau 2\partial_x^2}\Big(\bar u+i\lambda (s-t_n) \Pi_{\le N}\big[|u|^2\bar u\big]\Big) 
    \cdot\Big(\fe^{i\frac\tau 2\partial_x^2}u\Big)^2\Big]\,ds\notag\\
   &-2i \lambda  \int_{t_n}^{t_{n+1}} \fe^{-i(t_n+\frac\tau 2)\partial_x^2}\Pi_{\le N}
    \Big[\Big|\fe^{i\frac\tau 2\partial_x^2}u\Big|^2\cdot\fe^{i\frac\tau 2\partial_x^2}\Big(-i\lambda (s-t_n) \Pi_{\le N}\big[|u|^2u\big]\Big)\Big]\,ds\notag\\
   & 
   +R_n^3[u]+R_n^4[u],\label{Fn-tn+1-4}
\end{align}
where $R_n^4[u]$ is defined by 
\begin{align*}
R_n^4[u]
=&-i \lambda  \int_{t_n}^{t_{n+1}} \fe^{-i(t_n+\frac\tau 2)\partial_x^2}\Pi_{\le N}
    \Big[\fe^{-i\frac\tau 2\partial_x^2}\Big(\bar u+i\lambda (s-t_n) \Pi_{\le N}\big[|u|^2\bar u\big]\Big)\notag\\
   &\quad \cdot \Big(i\lambda (s-t_n)\Pi_{\le N}\fe^{i\frac\tau 2\partial_x^2}\big[|u|^2 u\big]\Big)^2\Big]\,ds\\
   &-2i \lambda  \int_{t_n}^{t_{n+1}} \fe^{-i(t_n+\frac\tau 2)\partial_x^2}\Pi_{\le N}
    \Big[  \Big|i\lambda (s-t_n)\Pi_{\le N}\fe^{i\frac\tau 2\partial_x^2}\big[|u|^2\bar u\big]\Big|^2
   \cdot  \fe^{i\frac\tau 2\partial_x^2}u\Big]\,ds.
\end{align*}

Note that \eqref{Fn-tn+1-4} implies that 
\begin{align}\label{Fn-tn+1-5}
G_n(v)=&  
     -i \lambda\tau \fe^{-it_n\partial_x^2}\fe^{-\frac i2\tau\partial_x^2}\Pi_{\le N}
        \Big[\big|\fe^{\frac i2\tau\partial_x^2}u\big|^2\fe^{\frac i2\tau\partial_x^2}u\Big]\notag\\
        &
        +\frac 12 \lambda^2\tau^2 \fe^{-it_n\partial_x^2}\fe^{-\frac i2\tau\partial_x^2}\Pi_{\le N}
        \Big[\fe^{-\frac i2\tau\partial_x^2}\Pi_{\le N}\big(|u|^2\bar u\big)\>\big(\fe^{\frac i2\tau\partial_x^2}u\big)^2\Big]\notag\\
        &
         -\lambda^2\tau^2 \fe^{-it_n\partial_x^2}\fe^{-\frac i2\tau\partial_x^2}\Pi_{\le N}
         \Big[\big|\fe^{\frac i2\tau\partial_x^2}u\big|^2\fe^{\frac i2\tau\partial_x^2}\Pi_{\le N}\big(|u|^2 u\big)\Big]\notag\\
         &+R_n^3[u]+R_n^4[u]\notag\\
        =&
     -i \lambda\tau \fe^{-it_n\partial_x^2}\fe^{-\frac i2\tau\partial_x^2}\Pi_{\le N}
        \Big[\big|\fe^{\frac i2\tau\partial_x^2}u\big|^2\fe^{\frac i2\tau\partial_x^2}u\Big]\notag\\
        &
        -\frac 12 \lambda^2\tau^2 \fe^{-it_n\partial_x^2}\fe^{-\frac i2\tau\partial_x^2}\Pi_{\le N}
        \Big[\big|\fe^{\frac i2\tau\partial_x^2}u\big|^4\fe^{\frac i2\tau\partial_x^2}u\Big]\notag\\
        &+R_n^3[u]+R_n^4[u]+R_n^5[u],
\end{align}
 where $R_n^5[u]$ is defined by 
 \begin{align}\label{def-R5}
  R_n^5[u]\triangleq
 &
        \frac 12 \lambda^2\tau^2\fe^{-\frac i2\tau\partial_x^2}\fe^{-it_n\partial_x^2}\Pi_{\le N}
        \Big[\Big(\fe^{-\frac i2\tau\partial_x^2}\Pi_{\le N}\big(|u|^2\bar u\big) -\big|\fe^{\frac i2\tau\partial_x^2}u\big|^2\fe^{-\frac i2\tau\partial_x^2}\bar u\Big)\big(\fe^{\frac i2\tau\partial_x^2}u\big)^2\Big]\notag\\
        &
         -\lambda^2\tau^2\fe^{-\frac i2\tau\partial_x^2}\fe^{-it_n\partial_x^2}\Pi_{\le N}
         \Big[\big|\fe^{\frac i2\tau\partial_x^2}u\big|^2 \Big(\fe^{\frac i2\tau\partial_x^2}\Pi_{\le N}\big(|u|^2 u\big)-\big|\fe^{\frac i2\tau\partial_x^2}u\big|^2\fe^{\frac i2\tau\partial_x^2}u\Big)\Big].
 \end{align}

{\bf Step 5: Converting to an exponential form.}

From \eqref{Fn-tn+1-5}, we have that 
 \begin{align}\label{Fn-tn+1-6}
\fe^{it_{n+1}\partial_x^2}&\big(v+G_n(v)\big)\notag\\
       =&\fe^{i\tau\partial_x^2}u
     -i \lambda\tau \fe^{\frac i2\tau\partial_x^2}\Pi_{\le N}
        \Big[\big|\fe^{\frac i2\tau\partial_x^2}u\big|^2\fe^{\frac i2\tau\partial_x^2}u\Big]
        -\frac 12 \lambda^2\tau^2\fe^{\frac i2\tau\partial_x^2}\Pi_{\le N}
        \Big[\big|\fe^{\frac i2\tau\partial_x^2}u\big|^4\fe^{\frac i2\tau\partial_x^2}u\Big]\notag\\
        &+\fe^{it_{n+1}\partial_x^2}\Big(R_n^3[u]+R_n^4[u]+R_n^5[u]\Big)\notag\\
        =&
        \fe^{\frac i2\tau\partial_x^2}\Pi_{\le N}
        \Big[\fe^{-i \lambda\tau \big|\fe^{\frac i2\tau\partial_x^2}u\big|^2}\fe^{\frac i2\tau\partial_x^2}u\Big]
        +\fe^{it_{n+1}\partial_x^2}\Big(R_n^3[u]+R_n^4[u]+R_n^5[u]+R_n^6[u]\Big),
\end{align}
 where $R_n^6[u]$ is defined by 
 \begin{align}\label{def-R3}
  R_n^6[u]=
 &
      -\fe^{-it_{n+1}\partial_x^2} \fe^{\frac i2\tau\partial_x^2}\Pi_{\le N}\Big[\Psi_1\Big(-i \lambda\tau \big|\fe^{\frac i2\tau\partial_x^2} u\big|^2\Big)\fe^{\frac i2\tau\partial_x^2}u\Big]
       ,\\ 
\Psi_1(x)=&\fe^x-(1+x+\frac12x^2)\notag.
 \end{align}
 
 In particular, denote that 
 \begin{align}\label{Def:Phi}
\Phi(u)= \fe^{\frac i2\tau\partial_x^2}\Pi_{\le N}
        \Big[\fe^{-i \lambda\tau \big|\fe^{\frac i2\tau\partial_x^2}u\big|^2}\fe^{\frac i2\tau\partial_x^2}u\Big],
 \end{align}
 then by \eqref{Fn-tn+1-6} and \eqref{def-G},  we obtain that 
 \begin{align}
\fe^{-it_{n+1}\partial_x^2}\Phi(u)
 =&
  v+ G_n(v)-\Big(R_n^3[u]+R_n^4[u]+R_n^5[u]+R_n^6[u]\Big)\notag\\
 =&v-i \lambda  \int_{t_n}^{t_{n+1}} \fe^{-is\partial_x^2}\Pi_{\le N}
    \Big[\Big|\fe^{is\partial_x^2}\Big(v-i\lambda (s-t_n)\fe^{-it_n\partial_x^2} \Pi_{\le N} \big[|u|^2u\big]\Big)\Big|^2\notag\\
   &\qquad\qquad \cdot\fe^{is\partial_x^2}\Big(v-i\lambda (s-t_n)\fe^{-it_n\partial_x^2} \Pi_{\le N}\big[|u|^2u\big]\Big)\Big]\,ds\notag\\
   &\quad- \Big(R_n^3[u]+R_n^4[u]+R_n^5[u]+R_n^6[u]\Big).\label{Phi-formula}
 \end{align}
 This formula will be used below. 
 
 Now by \eqref{solution-step1-low}, \eqref{Fn-Gn-vn} and \eqref{vu}, \eqref{Fn-tn+1-6} leads to the approximation of the low-frequency solution:  
 \begin{align}\label{solution-F-u-low}
\Pi_{\le N}u(t_{n+1})
= &\fe^{\frac i2\tau\partial_x^2}\Pi_{\le N}
        \Big[\fe^{-i \lambda\tau \big|\fe^{\frac i2\tau\partial_x^2}\Pi_{\le N}u(t_n)\big|^2}\fe^{\frac i2\tau\partial_x^2}\Pi_{\le N}u(t_n)\Big]
        +\fe^{it_{n+1}\partial_x^2}\Pi_{\le N}R_n^1(t_{n+1})\notag\\
        &+\fe^{it_{n+1}\partial_x^2}\Big(R_n^2+R_n^3[\Pi_{\le N}u]+R_n^4[\Pi_{\le N}u]+R_n^5[\Pi_{\le N}u]+R_n^6[\Pi_{\le N}u]\Big).
\end{align}

For the  high-frequency solution, we apply \eqref{solution-step1-high}  and write 
\begin{align}\label{u-eqs-btN}
 \Pi_{> N}u(t_{n+1})=&\Pi_{> N}\fe^{ i\tau\partial_x^2}\fe^{-4\pi i\lambda M_0\tau}u(t_n)+\fe^{it_{n+1}\partial_x^2}\Pi_{> N} R_n^1(t_{n+1})+R_n^7\big[\Pi_{> N}u(t_n)\big],
\end{align}
where 
$$
R_n^7[u]=-\Psi_2(-4\pi i\lambda M_0\tau) \fe^{ i\tau\partial_x^2}u,\quad 
\Psi_2(x)=\fe^{x}-(1+x).
$$

Since 

The numerical scheme can be defined by dropping the defect terms $R_n^j, j=1,\cdots,7$. 
Namely, for $n=0,1\ldots, L - 1$,
\begin{align}
\Pi_{\le N} u^{n+1}=&\Pi_{\le N}\fe^{i\frac\tau2\partial_x^2}\mathcal N_\tau\big[\fe^{i\frac\tau2\partial_x^2}\Pi_{\le N}u^n\big];
\label{u-n-ltN}\\
\Pi_{> N} u^{n+1}=&\Pi_{>N}\big[\fe^{i\tau\partial_x^2-4\pi i\lambda M_0\tau}u^n\big].\label{u-n-btN}
\end{align}

\section{Continuation method and framework of the proof}\label{sec:Continuation}

For the scheme defined in \eqref{u-n-ltN}--\ref{u-n-btN}, we have the following more general results.

\begin{proposition}\label{prop:main}
Let $\gamma\in (0,2]$. If $u^0\in H^\gamma(\T)$, then there exist positive constants $\tau_0$  and $C$ such that for $\tau\leq\tau_0$ and $\tau N\le 1$, the numerical solution given by  \eqref{u-n-ltN}--\ref{u-n-btN} has the following error bound:
\begin{equation}\label{error-N-1}
\max_{1\le n\le L}  \|u(t_n)-u^{n}\|_{L^2}
  \le C\Big( N^{-\min\{2\gamma,\alpha(\gamma)\}}+\tau N^{-\gamma+}+\tau^2 N^{4-\gamma}\Big), 
\end{equation} 
where the constants $\tau_0$ and $C$ depend only on $T$ and $\|u^0\|_{H^\gamma}$. 
\end{proposition}

Before performing the proof of Proposition \ref{prop:main}, let us indicate how it implies Theorem \ref{the:main} for \eqref{NuSo-NLS-1}.

\begin{proof}[Proof of Theorem \ref{the:main} by \eqref{NuSo-NLS-1}]
Note that when $\gamma\in (0,1)$, 
$$
N^{-\min\{2\gamma,\alpha(\gamma)\}}=N^{-2\gamma}.
$$
Then \eqref{error-N-1} implies that 
\begin{equation*} 
\max_{1\le n\le L}  \|u(t_n)-u^{n}\|_{L^2}
  \le C\Big( N^{-2\gamma}+\tau N^{-\gamma+}+\tau^2 N^{4-\gamma}\Big), 
\end{equation*} 
Now we choose 
$$
\tau^{-\frac{2\gamma}{4+\gamma}}, 
$$
then it gives the desired estimate in \eqref{error-tau-1}. 

Note that when $\gamma\in [1,2]$, 
$$
N^{-\min\{2\gamma,\alpha(\gamma)\}}=N^{-\alpha(\gamma)}=N^{-(1+\gamma)+}.
$$
Then choosing $N=\tau^{-\frac25+}$ such that $N^{-\alpha(\gamma)}=\tau^2 N^{4-\gamma}$, 
\eqref{error-N-1} implies the desired estimate in \eqref{error-tau-2}.  This gives the conclusion in Theorem \ref{the:main} for \eqref{NuSo-NLS-1}. 
\end{proof}

Hence, it reduces to the proof of Proposition \ref{prop:main}.

\subsection{Continuation method: Low-frequency component}


In this subsection, we construct a continuous solution $\mathcal U_{\le N}(t)$ such that $\mathcal U_{\le N}(t_n)=\Pi_{\le N} u^n$ for $n=0,1,\cdots, L-1$. 
To do this, we denote $v^n=\fe^{-it_n\partial_x^2}u^n$. Then applying  \eqref{Def:Phi}, \eqref{Phi-formula} and \eqref{u-n-ltN}, we obtain that 
\begin{align*}
\Pi_{\le N}v^{n+1}=  
&\Pi_{\le N}v^n\\
&-i \lambda  \int_{t_n}^{t_{n+1}} \fe^{-is\partial_x^2}\Pi_{\le N}
    \Big[\Big|\fe^{is\partial_x^2}\Pi_{\le N}\Big(v^n-i\lambda (s-t_n)\fe^{-it_n\partial_x^2} \big[|\Pi_{\le N}u^n|^2\Pi_{\le N}u^n\big]\Big)\Big|^2\\
   &\qquad \cdot\fe^{is\partial_x^2}\Pi_{\le N}\Big(v^n-i\lambda (s-t_n)\fe^{-it_n\partial_x^2} \big[|\Pi_{\le N}u^n|^2\Pi_{\le N}u^n\big]\Big)\Big]\,ds+\Upsilon_n^1,
\end{align*}
where we denote 
\begin{align}\label{def:Rn1}
\Upsilon_n^1\triangleq -\Big(R_n^3[\Pi_{\le N}u^n]+R_n^4[\Pi_{\le N}u^n]+R_n^5[\Pi_{\le N}u^n]+R_n^6[\Pi_{\le N}u^n]\Big).
\end{align}

Accordingly, we define 
\begin{align}
\mathcal V_{\le N}(t)
 \triangleq 
 &\Pi_{\le N}v^n-i \lambda  \int_{t_n}^t \fe^{-is\partial_x^2}\Pi_{\le N}
    \Big[\Big|\fe^{is\partial_x^2}\Pi_{\le N}\Big(v^n-i\lambda (s-t_n)\fe^{-it_n\partial_x^2} \big[|\Pi_{\le N} u^n|^2\Pi_{\le N} u^n\big]\Big)\Big|^2
    \notag\\
   &\quad \cdot\fe^{is\partial_x^2}\Pi_{\le N}\Big(v^n-i\lambda (s-t_n)\fe^{-it_n\partial_x^2} \big[|\Pi_{\le N} u^n|^2\Pi_{\le N} u^n\big]\Big)\Big]\,ds
   +\frac{t-t_n}{\tau}\Upsilon_n^1, \label{V-ltN}
\end{align}
for $t\in [t_n,t_{n+1}]$.  Note that $\mathcal V_{\le N}$ is continuous in time and  satisfies that 
$$
\mathcal V_{\le N}(t_n)=\Pi_{\le N} v^n;\quad 
\mathcal V_{\le N}(t_{n+1})=\Pi_{\le N} v^{n+1}.
$$

Further, we denote 
$$
 \mathcal U_{\le N}(t)\triangleq \fe^{it\partial_x^2}\mathcal V_{\le N}(t),
 $$
 and 
 \begin{align}\label{def:rn} 
 r^n(t) \triangleq\mathcal V_{\le N}(t)-\Pi_{\le N}\Big(v^n-i\lambda (t-t_n)\fe^{-it_n\partial_x^2} \big[|\Pi_{\le N} u^n|^2\Pi_{\le N} u^n\big]\Big).
\end{align}
Then we rewrite $\mathcal V_{\le N}(t)$ as
\begin{align}
\mathcal V_{\le N}(t)
 = 
 &\Pi_{\le N}v^n-i \lambda  \int_{t_n}^{t} \fe^{-is\partial_x^2}\Pi_{\le N}
    \Big[\big|\mathcal U_{\le N}(s)\big|^2\mathcal U_{\le N}(s)\Big]\,ds+\frac{t-t_n}{\tau} \Upsilon_n^1+\Upsilon_n^2(t), \label{V-eqs}
\end{align}
where 
\begin{align}\label{Rnt}
\Upsilon_n^2(t)\triangleq &i \lambda  \int_{t_n}^t \fe^{-is\partial_x^2}\Pi_{\le N}
    \Big[\fe^{-is\partial_x^2}\overline{r^n(s)}\cdot \Big(\mathcal U_{\le N}(s)\Big)^2\Big]\,ds\notag\\
    &+2i \lambda  \int_{t_n}^t \fe^{-is\partial_x^2}\Pi_{\le N}\Big[
    \big|\mathcal U_{\le N}(s)\big|^2 \cdot \fe^{is\partial_x^2} r^n(s)-\big|\fe^{is\partial_x^2} r^n(s)\big|^2 \cdot \mathcal U_{\le N}(s)\Big]\,ds\notag\\
     &-i \lambda  \int_{t_n}^t \fe^{-is\partial_x^2}\Pi_{\le N}
    \Big[\overline{\mathcal U_{\le N}(s)}\cdot \big(\fe^{is\partial_x^2}r^n(s)\big)^2-\big|\fe^{is\partial_x^2}r^n(s)\big|^2\fe^{is\partial_x^2}r^n(s)\Big]\,ds.
\end{align}
Taking the summation, we obtain that 
\begin{align}
\mathcal V_{\le N}(t)
 = 
 &\Pi_{\le N}v^0-i \lambda  \int_0^{t} \fe^{-is\partial_x^2}\Pi_{\le N}
    \Big[\big|\mathcal U_{\le N}(s)\big|^2\mathcal U_{\le N}(s)\Big]\,ds+
    \Upsilon(t), \label{V-sum}
\end{align}
where 
\begin{align}\label{def:R}
\Upsilon(t)\triangleq
\frac{t-t_n}{\tau} \Upsilon_n^1+\Upsilon_n^2(t)+\sum\limits_{j=0}^{n-1}\big( \Upsilon_j^1+\Upsilon_j^2(t_{j+1})\big).
\end{align}
This leads that  for $t\in [t_n,t_{n+1}]$,
\begin{align}\label{U-eqs}
\mathcal U_{\le N}(t)
 = 
 &\fe^{it\partial_x^2}\Pi_{\le N}u^0-i \lambda  \int_0^{t} \fe^{i(t-s)\partial_x^2}\Pi_{\le N}
    \Big[\big|\mathcal U_{\le N}(s)\big|^2\mathcal U_{\le N}(s)\Big]\,ds+\fe^{it\partial_x^2}\Upsilon(t).
\end{align}


\subsection{Framework of the proof: Low-frequency component}
Note that by \eqref{solution-step1-low} (replacing $t_{n+1}$ by $t$) and Taking the summation, we have that 
\begin{align}\label{u-eqs}
\Pi_{\le N} u(t)
 = 
 &\fe^{it\partial_x^2}\Pi_{\le N}u_0-i \lambda  \int_0^t \fe^{i(t-s)\partial_x^2}\Pi_{\le N}
    \big[|\fe^{is\partial_x^2}\Pi_{\le N}v(s)|^2\,\fe^{is\partial_x^2}\Pi_{\le N}v(s)\big]\,ds\notag\\
     &+\fe^{it\partial_x^2}\Pi_{\le N}\Upsilon_0(t),
\end{align}
where 
\begin{align}\label{def:CR0}
\Upsilon_0(t)\triangleq 
\sum\limits_{j=0}^{n-1}\Upsilon_j^1(t_{j+1})+\Upsilon_n^1(t).
\end{align}

Denote 
$$
\mathcal E_{\le N}(t)\triangleq\mathcal U_{\le N}(t)-\Pi_{\le N}u(t),
$$
then by \eqref{U-eqs} and  \eqref{u-eqs}, we obtain that 
\begin{align}
\mathcal E_{\le N}(t)
 = 
 &\fe^{it\partial_x^2}\mathcal E_{\le N}(0)-i \lambda  \int_0^{t} \fe^{i(t-s)\partial_x^2}\Pi_{\le N}
    \Big[\big|\mathcal U_{\le N}(s)\big|^2\mathcal U_{\le N}(s)-\big|\Pi_{\le N}u(s)\big|^2\Pi_{\le N}u(s)\Big]\,ds\notag\\
    &+\fe^{it\partial_x^2}\Big(\Upsilon(t)-\Pi_{\le N}\Upsilon_0(t)\Big).\label{def-E}
\end{align}

Now for any fixed $T>0$, we define the working space and denote 
\begin{align}\label{def:XT}
\mathcal X([0,T])\triangleq  &A(N,\tau)^{-1}\Big(\big\|\mathcal E_{\le N}\big\|_{L^\infty_tL^2_x\cap L^4_{tx}([0,T])}+\big\|J^{0-}\mathcal E_{\le N}\big\|_{L^6_{tx}([0,T])}\Big);\notag\\
\mathcal Y([0,T])\triangleq  &\tau^\frac16 \big\|J^{\gamma-}\Pi_{\le N}u^n\big\|_{l^6_nL^6_{x}},
\end{align}
where 
\begin{align}\label{def-ANtau}
A(N,\tau)\triangleq & N^{-\min\{2\gamma,\alpha(\gamma)\}}+\tau N^{-\gamma+}+\tau^2 N^{4-\gamma}.
\end{align}

\subsection{Continuation method: High-frequency component}

In this subsection, we construct a continuous solution $\mathcal U_{> N}(t)$ such that $\mathcal U_{> N}(t_n)=\Pi_{> N} u^n$ for $n=0,1,\cdots, L-1$. 
Again, we denote $v^n=\fe^{-it_n\partial_x^2}u^n$. Note that 
$$
\fe^{-4\pi i\lambda M_0\tau}u=(1-4\pi i\lambda M_0\tau)u+\Psi_2(-4\pi i\lambda M_0\tau)u.
$$
Then we have that 
\begin{align*} 
 \Pi_{> N}v^{n+1}=&\Pi_{> N}v^n-4\pi i\lambda M_0\tau \Pi_{> N}v^n+\Psi_2(-4\pi i\lambda M_0\tau)\Pi_{> N}v^n.
\end{align*}
Accordingly, we define that for any $t\in [t_n,t_{n+1}]$, 
\begin{align} \label{def:V-btN}
 \mathcal V_{> N}(t)= &\Pi_{> N}v^n-4\pi i\lambda M_0(t-t_n) \Pi_{> N}v^n+\frac{t-t_n}{\tau}\Psi_2(-4\pi i\lambda M_0\tau)\Pi_{> N}v^n.
\end{align}
Then we have that $\mathcal V_{> N}(t)$ is a continuous function in $\R$ and $ \mathcal V_{> N}(t_n)=\Pi_{> N} v^n$, for any $n=0, 1,\cdots, L-1$. 
Furthermore, we  rewrite it as 
\begin{align}\label{V-eqs-btN}
 \mathcal V_{> N}(t)= &\Pi_{> N}v^n-4\pi i\lambda M_0\int_{t_n}^t \mathcal V_{> N}(s)\,ds+\Upsilon_n^3(t),
\end{align}
where 
\begin{align}\label{def:R3n}
\Upsilon_n^3(t)=4\pi i\lambda M_0\int_{t_n}^t \big(\mathcal V_{> N}(s)-\Pi_{> N}v^n\big)\,ds+\frac{t-t_n}{\tau}\Psi_2(-4\pi i\lambda M_0\tau)\Pi_{> N}v^n.
\end{align}
Therefore, taking the summation, it implies that  for any $t\in [t_n,t_{n+1}]$, 
\begin{align}\label{V-eqs-btN}
 \mathcal V_{> N}(t)= &\Pi_{> N}v^0-4\pi i\lambda M_0\int_0^t \mathcal V_{> N}(s)\,ds+\Upsilon_3(t),
\end{align}
where 
$$
\Upsilon_3(t)\triangleq\Upsilon_n^3(t)+\sum\limits_{j=0}^{n-1}\Upsilon_j^3(t_{j+1}).
$$

\subsection{Framework of the proof: High-frequency component}
By \eqref{solution-step1-high} (replacing $t_{n+1}$ by $t$), we have that   for any $t\in [t_n,t_{n+1}]$, 
$$
\Pi_{> N}v(t)=\Pi_{> N}v(t_n) -4\pi i\lambda M_0 \int_{t_n}^t \Pi_{>N}v(s)\,ds+\Upsilon_n^4(t),
$$
where 
$$
\Upsilon_n^4(t)\triangleq 4\pi i\lambda M_0 \int_{t_n}^t \Pi_{>N}\big(v(s)-v(t_n)\big)\,ds+\Pi_{> N} R_n^1(t_{n+1}).
$$
Hence,  taking the summation, it gives that 
\begin{align}\label{v-eqs-btN}
\Pi_{> N}v(t)=\Pi_{> N}v^0 -4\pi i\lambda M_0 \int_0^t \Pi_{>N}v(s)\,ds+\Upsilon_4(t),
\end{align}
where by the definition \eqref{def:CR0},
$$
\Upsilon_4(t)\triangleq 4\pi i\lambda M_0\sum\limits_{j=0}^{n-1}\int_{t_j}^{t_{j+1}} \Pi_{>N}\big(v(s)-v(t_j)\big)\,ds+4\pi i\lambda M_0 \int_{t_n}^t \Pi_{>N}\big(v(s)-v(t_n)\big)\,ds+\Pi_{> N} \Upsilon_0(t).
$$
Then combining with \eqref{V-eqs-btN} and \eqref{v-eqs-btN}, we obtain that 
\begin{align}
\mathcal V_{> N}(t)-\Pi_{> N}v(t)=-4\pi i\lambda M_0\int_{t_n}^t \Big( \mathcal V(s)- \Pi_{>N}v(s)\Big)\,ds+\Upsilon_3(t)-\Upsilon_4(t).
\end{align}
Denote that 
$$
\mathcal U_{> N}(t)\triangleq \fe^{it\partial_x^2}\mathcal V_{> N}(t);\quad 
\mathcal E_{> N}(t)= \mathcal U_{> N}(t)-\Pi_{> N}u(t),
$$
then
\begin{align}\label{E-eqs-btN}
\mathcal E_{> N}(t)=-4\pi i\lambda M_0\int_0^t \fe^{i(t-s)\partial_x^2}\mathcal E_{> N}(s)\,ds+\fe^{it\partial_x^2}\Big(\Upsilon_3(t)-\Upsilon_4(t)\Big).
\end{align} 
Then it follows that for any $t\in [0,T]$, 
\begin{align*} 
\big\|\mathcal E_{> N}(t)\big\|_{L^2}\le 4\pi |\lambda| M_0\int_0^t \big\|\mathcal E_{> N}(s)\big\|_{L^2}\,ds+ 
\big\|\Upsilon_3(t)\big\|_{L^\infty_tL^2_x([0,T])}+\big\|\Upsilon_4(t)\big\|_{L^\infty_tL^2_x([0,T])}.
\end{align*}
Therefore, by Gronwall's inequlity, it infers that there exists some constant $C=C(T)>0$, such that 
\begin{align}\label{E-high}
\big\|\mathcal E_{> N}(t)\big\|_{L^\infty_tL^2_x([0,T])}\le C\Big[ 
\big\|\Upsilon_3(t)\big\|_{L^\infty_tL^2_x([0,T])}+\big\|\Upsilon_4(t)\big\|_{L^\infty_tL^2_x([0,T])}\Big].
\end{align}

%
%


\section{Analyzing the errors}
\label{section:errors}

%
%


\subsection{Preliminary estimates}

Let $\gamma>0$ and $u_0\in H^\gamma$.  Then by Lemma \ref{lem:global-theory}, for any fixed $\delta_0>0$, there exists $T_0>0$ such that 
\begin{align}\label{solution-est-small}
\big\|J^\gamma u\big\|_{L^4_{tx}([0,T_0])}
+\big\|J^{\gamma-} u\big\|_{L^6_{tx}([0,T_0])}
\le \delta_0,
\end{align}
where $\delta_0$ is a small constant determined later. Moreover, we set $N$ such that 
\begin{align}\label{tauN2-small}
N^{\gamma}A(N,\tau)\le \delta_0.
\end{align}
Note that the restriction is not necessary for our analysis by replacing a smallness assumption of the spacetime norm of $\mathcal E_{\le N}$, however, we decide to keep this restriction to avoid making our approach vague and complex.

A direct consequence of \eqref{def:XT}, \eqref{solution-est-small} and Lemma \ref{lem:global-theory}  is  
\begin{lemma}\label{lem:JU}
Let $\gamma>0, T\in (0,T_0], N\ge 1$ and $u_0\in H^\gamma$, then there exists  some $C=C(T,\|u_0\|_{H^\gamma})>0$, such that 
\begin{align}\label{est:PN-U}
\big\|J^\gamma\mathcal U_{\le N}\big\|_{L^4_{tx}([0,T])}+\big\|J^{\gamma-}\mathcal U_{\le N}\big\|_{L^6_{tx}([0,T])}
\le C\delta_0\big(1+\mathcal X([0,T])\big).
\end{align}
\end{lemma}
\begin{proof}
Noting that $A(N,\tau)\le \delta_0 N^{-\gamma}$,  we have that 
\begin{align*}
\big\|J^\gamma\mathcal E_{\le N}\big\|_{L^4_{tx}([0,T])}&+\big\|J^{\gamma-}\mathcal E_{\le N}\big\|_{L^6_{tx}([0,T])}\\
\lesssim 
&
N^\gamma\Big( \big\|\mathcal E_{\le N}\big\|_{L^4_{tx}([0,T])}+\big\|J^{0-}\mathcal E_{\le N}\big\|_{L^6_{tx}([0,T])}\Big)
\\
\lesssim &
N^\gamma A(N,\tau) \mathcal X([0,T])\\
\lesssim & \delta_0\mathcal X([0,T]).
\end{align*}
This further gives that   
\begin{align*}
&\big\|J^\gamma\mathcal U_{\le N}\big\|_{L^4_{tx}([0,T])}+\big\|J^{\gamma-}\mathcal U_{\le N}\big\|_{L^6_{tx}([0,T])}\\
\lesssim 
&
\big\|J^\gamma\mathcal E_{\le N}\big\|_{L^4_{tx}([0,T])}+\big\|J^{\gamma-}\mathcal E_{\le N}\big\|_{L^6_{tx}([0,T])}\\
&\quad+ \big\|J^\gamma\Pi_{\le N}u\big\|_{L^4_{tx}([0,T])}+\big\|J^{\gamma-}\Pi_{\le N}u(t)\big\|_{L^6_{tx}([0,T])}\\
\lesssim &
\delta_0\big(1+\mathcal X([0,T])\big).
\end{align*}
This finishes the proof of the lemma.
\end{proof}

\subsection{Estimates on $\Upsilon_0(t)$}

\begin{proposition}\label{prop:R0}
Let $\gamma>0, T>0, N\ge 1$, $\alpha(\gamma)$ be defined in \eqref{def:alpha}, and $u_0\in H^\gamma$, then there exists  some $C=C(T,\|u_0\|_{H^\gamma})>0$, such that 
$$
\big\|\Upsilon_0(t)\big\|_{L^\infty_tH^\gamma_x([0,T])}
\le C\Big(N^{-\min\{2\gamma,\alpha(\gamma)\}}+\tau N^{-\gamma+}\Big).
$$
\end{proposition}
\begin{proof}
Note that 
\begin{subequations}
\begin{align}
\Upsilon_0(t)
=& -i \lambda \int_0^t\fe^{-is\partial_x^2} \Pi_{>N}\Big[\fe^{-is\partial_x^2}\Pi_{\le N}\bar v(s)\cdot \fe^{is\partial_x^2} \Pi_{\le N}v(s)\cdot \fe^{is\partial_x^2} \Pi_{\le N}v(s)\Big]\,ds\label{R1-a}\\
& -i \lambda \int_0^t\fe^{-is\partial_x^2} \Big[\fe^{-is\partial_x^2}\Pi_{> N}\bar v(s)\cdot \fe^{is\partial_x^2} \Pi_{\le N}v(s)\cdot \fe^{is\partial_x^2}\Pi_{\le N} v(s)\Big]\,ds\\
& -i \lambda \int_0^t\fe^{-is\partial_x^2} \Big[\fe^{-is\partial_x^2}\bar v(s)\cdot \fe^{is\partial_x^2} \Pi_{> N}v(s)\cdot \fe^{is\partial_x^2}\Pi_{\le N} v(s)\Big]\,ds\notag\\
& 
\quad +i \lambda \int_0^t \Pi_{\le N} v(s)\int_\T|\Pi_{> N}v(s)|^2\,dx \,ds
+i \lambda \int_0^t \Pi_{> N} v(s)\int_\T|\Pi_{\le N}v(s)|^2\,dx \,ds\\
& -i \lambda \int_0^t\fe^{-is\partial_x^2} \Big[\fe^{-is\partial_x^2}\bar v(s)\cdot \fe^{is\partial_x^2} v(s)\cdot \fe^{is\partial_x^2}\Pi_{> N} v(s)\Big]\,ds\notag\\
& 
\quad+i \lambda \int_0^t  v(s)\int_\T|\Pi_{> N}v(s)|^2\,dx \,ds
+i \lambda \int_0^t \Pi_{> N} v(s)\int_\T|v(s)|^2\,dx \,ds\label{R1-d}\\
  & -2i \lambda \int_0^t \Pi_{\le N} v(s)\int_\T|\Pi_{> N}v(s)|^2\,dx \,ds   \label{R1-e}\\
  &+ 4\pi i \lambda M_0  \Big( \sum\limits_{j=0}^{n-1}  \int_{t_j}^{t_{j+1}} \Pi_{>N}\>\big[v(s)-v(t_j)\big]\,ds
  + \int_{t_n}^{t} \Pi_{>N}\>\big[v(s)-v(t_n)\big]\,ds\Big).\label{R1-f}
\end{align}
\end{subequations}

First, by Lemmas \ref{lem:global-theory} and \ref{lem:est-vt}, we have that for some $C=C(T,\|u_0\|_{H^\gamma})>0$, 
\begin{align}\label{est:X-T}
\|v\|_{X^{\gamma-}([0,T])}+\left\| \fe^{it\partial_x^2}J^{\gamma-} v\right\|_{L^3_tL^6_x([0,T])}\le C.
\end{align}
Note that \eqref{R1-a}--\eqref{R1-d} obey the structure of \eqref{def-T} with $T_j=\Pi_{\le N}$ or $\Pi_{>N}$.  Therefore, applying Lemma \ref{lem:tri-est}, we obtain that 
\begin{align*}
\big\|\eqref{R1-a}\big\|_{L^2}+\cdots+\big\|\eqref{R1-d}\big\|_{L^2}
\le  CN^{-\min\{3\gamma,\alpha(\gamma)\}+}.
\end{align*}

By Bernstein's inequality, we have that 
\begin{align*}
\big\|\eqref{R1-e}\big\|_{L^2}
\le  CN^{-2\gamma}.
\end{align*}

For \eqref{R1-f},  by Lemma \ref{lem:est-vt-2}, 
\begin{align*}
\big\|\eqref{R1-f}\big\|_{L^2}\lesssim &\tau N^{-\gamma+} \sum\limits_{n=1}^{L-1}\big\|J^{\gamma-}\big(v(t_n+s)-v(t_n)\big)\big\|_{L^\infty_sL^2_{x}([0,\tau]\times \T)}
\le C\tau N^{-\gamma+}.
\end{align*}

This finishes the proof of the proposition.
\end{proof}

\noindent

In the following two subsections, we present the estimate on $\Upsilon(t)$.

\begin{proposition}\label{prop:R}
Let $\gamma>0, T>0, N\ge 1$,  and $u_0\in H^\gamma$, then then there exists  $C=C(T,\|u_0\|_{H^\gamma})>0$, such that 
$$
\big\|\Upsilon(t)\big\|_{L^\infty_tH^\gamma_x([0,T])}
\le C  \tau^2 N^{4-\gamma+}\Big(\mathcal Y([0,T])^3+\mathcal Y([0,T])^{12}\Big).
$$
\end{proposition}

By \eqref{def:R}, $\Upsilon$ is a combination of $\Upsilon_n^1$ and $\Upsilon_n^2$, the proposition is followed from the estimates given in the subsections below.

\subsection{Estimates on $\Upsilon_n^1$} 
The main estimate in this subsection is 
\begin{lemma}\label{lem:Rn1}
Let  $\gamma>0, N\ge 1$ and $u_0\in H^\gamma$, then there exists  $C=C(T,\|u_0\|_{H^\gamma})>0$, such that 
$$
\sum\limits_{n =0}^{L-1}\big\|\Upsilon_n^1\big\|_{H^\gamma_x}
\le C  \tau^2 N^{4-\gamma+}\Big(\mathcal Y([0,T])^3+\mathcal Y([0,T])^6\Big).
$$
\end{lemma}
From \eqref{def:Rn1}, it reduces to the estimation on $R_n^j, j=3,\cdots, 6$.
First, we give the estimates on $R_n^3$.
\begin{lemma}\label{lem:R3}
Let $\gamma>0$,  then there exists $C>0$ independent of $\tau, N$ and $u^n$, such that 
$$
\sum\limits_{n =0}^{L-1}\big\|R_n^3[\Pi_{\le N}u^n]\big\|_{L^2_x}
\le C  \tau^2 N^{4-\gamma+}\mathcal Y([0,T])^3.
$$
\end{lemma}
\begin{proof} 
We write $u=\Pi_{\le N}u^n$ for short. 
Note that 
\begin{align*}
r_0(s)=&\Big(\fe^{isk^2}-\fe^{i\frac\tau2k^2}\Big)+\fe^{isk^2}\Big(\fe^{isk_1^2}-\fe^{i\frac\tau2k_1^2}\Big)\\
&\quad +\fe^{is(k^2+k_1^2)}\Big(\fe^{-isk_2^2}-\fe^{-i\frac\tau2k_2^2}\Big)
+\fe^{is(k^2+k_1^2-k_2^2)}\Big(\fe^{-isk_3^2}-\fe^{-i\frac\tau2k_3^2}\Big).
\end{align*}
We only focus on the first piece of $r_0(s)$ above, which allows us to rewrite 
\begin{align*}
R_n^3[u]= & -i \lambda \int_0^\tau\fe^{-it_n\partial_x^2}\Big(\fe^{-is\partial_x^2}-\fe^{-i\frac\tau2\partial_x^2}\Big) 
        \Big[\Big|\fe^{i\frac\tau2\partial_x^2}\Big(u-i\lambda s\big(|u|^2 u\big)\Big)\Big|^2
        \\
        &\qquad \cdot \fe^{i\frac\tau2\partial_x^2}\Big(u-i\lambda s\big(|u|^2u\big)\Big)\Big]\,ds\\
        &+\mbox{other similar terms}.
\end{align*}
Here the ``other similar terms'' are the terms can be treated in similar manner. 
In particular, we expand 
$$
\fe^{isk^2}\Big(\fe^{isk_1^2}-\fe^{i\frac\tau2k_1^2}\Big)=\Big(\fe^{isk_1^2}-\fe^{i\frac\tau2k_1^2}\Big)+\Big(\fe^{isk^2}-\fe^{i\frac\tau2k^2}\Big)\Big(\fe^{isk_1^2}-\fe^{i\frac\tau2k_1^2}\Big).
$$
Then we can deal with it by the following same manner.

For convenience, we denote the operator $\mathcal T_s$ as 
$$
\mathcal T_s f=\Big(\fe^{-is\partial_x^2}-1+is\partial_x^2\Big)f.
$$
Then by Lemma \ref{lem:multiplier}, we have that for $1<p<\infty$, 
\begin{align}\label{est:opr-Ts}
\big\|\mathcal T_s \Pi_{\le N}f\big\|_{L^p}\lesssim s^2 \|J^4 \Pi_{\le N}f\|_{L^p},
\end{align}
and 
\begin{align}\label{est:opr-e}
\Big\|\big(\fe^{-is\partial_x^2}-1\big) \Pi_{\le N}f\Big\|_{L^p}\lesssim s \|J^2 \Pi_{\le N}f\|_{L^p}.
\end{align}
Moreover, 
$$
\fe^{-is\partial_x^2}-\fe^{-i\frac\tau2\partial_x^2}=\fe^{-i\frac\tau2\partial_x^2}\Big(-i(s-\frac\tau 2)\partial_x^2+\mathcal T_{s-\frac\tau2}\Big).
$$
Then we can further rewrite 
\begin{align*}
R_n^3[u]= & -i \lambda \int_0^\tau\fe^{-it_n\partial_x^2}\>\fe^{-i\frac\tau2\partial_x^2}\Big(-i(s-\frac\tau 2)\partial_x^2+\mathcal T_{s-\frac\tau2}\Big) 
        \\
        &\qquad 
       \cdot \Big[\Big| \fe^{i\frac\tau2\partial_x^2}\Big(u-i\lambda s\big(|u|^2u\big)\Big) \Big|^2
        \fe^{i\frac\tau2\partial_x^2}\Big(u-i\lambda s \big(|u|^2u\big)\Big)\Big]\,ds\\
        &+\mbox{other similar terms}.
\end{align*}
Therefore, using \eqref{est:opr-Ts} and \eqref{est:opr-e}, H\"older and Bernstein's inequalities,  and noting that 
$$
\int_0^\tau (s-\frac\tau2)\,ds=0,
$$
$\|u^n\|_{L^2_x}\le \|u_0\|_{L^2_x}$, we have that 
\begin{align*}
\big\|R_n^3[u]\big\|_{L^2}
\lesssim 
&\tau^3 N^{4-\gamma+}\Big(1+\tau \|u\|_{L^\infty_x}^2\Big)^3 \|J^{\gamma-} u\|_{L^6_x}^3\\
\lesssim 
&\tau^3 N^{4-\gamma+}\Big(1+\tau N \|u\|_{L^2_x}^2\Big)^3\|J^{\gamma-}u\|_{L^6_x}^3\\
\lesssim 
&\tau^3 N^{4-\gamma+}\big(1+\tau N\big)^3\|J^{\gamma-}u\|_{L^6_x}^3\\
\lesssim 
&\tau^3 N^{4-\gamma+}\big(1+\tau N^2\big)^3\|J^{\gamma-}u\|_{L^6_x}^3.
\end{align*}
Note that by \eqref{tauN2-small}, we have that $\tau N^2\le 1$, thus 
\begin{align*}
\big\|R_n^3[u]\big\|_{L^2}
\lesssim 
&\tau^3 N^{4-\gamma+}\|J^{\gamma-}u\|_{L^6_x}^3.
\end{align*}

Now taking the summation,  by Cauchy-Schwarz's inequality,
we get that 
\begin{align*}
\sum\limits_{n =0}^{L-1}\big\|R_n^3[\Pi_{\le N}u^n]\big\|_{L^2_x}
\lesssim 
&\tau^3 N^{4-\gamma+}\sum\limits_{n =0}^{L-1} \big\|J^{\gamma-} \Pi_{\le N}u^n\big\|_{L^6_x}^3\\
\lesssim 
&\tau^\frac52 N^{4-\gamma+}\|J^{\gamma-}\Pi_{\le N}u^n\|_{l^6_nL^6_x}^3\\
\lesssim &
 \tau^2 N^{4-\gamma+}\mathcal Y([0,T])^3.
\end{align*}
This finishes the proof of the lemma.
\end{proof}

Second, we give the estimates on $R_n^4$.
\begin{lemma}\label{lem:R4}
Let $\gamma>0$,  then  there exists  $C>0$ independent of $\tau, N$ and $u^n$, such that 
$$
\sum\limits_{n =0}^{L-1}\big\|R_n^4[\Pi_{\le N}u^n]\big\|_{L^2_x}
\le C \tau^2 N \Big(1+\tau N\Big)\mathcal Y([0,T])^6.
$$
\end{lemma}
\begin{proof} 
By H\"older and Bernstein's inequalities, and noting that $\|u^n\|_{L^2_x}\le \|u_0\|_{L^2_x}$, 
we have that 
\begin{align*}
\big\|R_n^4[\Pi_{\le N}u^n]\big\|_{L^2}
\lesssim 
&\tau^3 \Big( \big\|\Pi_{\le N}u^n\big\|_{L^{14}_x}^7+\tau \big\|\Pi_{\le N}u^n\big\|_{L^{18}_x}^9\Big)\\
\lesssim 
&\tau^3 \Big(N\big\|\Pi_{\le N}u^n\big\|_{L^2_x}\big\|\Pi_{\le N}u^n\big\|_{L^{6}_x}^6+\tau N^2 \big\|\Pi_{\le N}u^n\big\|_{L^2_x}^3\big\|\Pi_{\le N}u^n\big\|_{L^{6}_x}^6\Big)\\
\lesssim 
&\tau^3 \big(N+\tau N^2 \big)\big\|\Pi_{\le N}u^n\big\|_{L^{6}_x}^6.
\end{align*}
This implies that 
\begin{align*}
\sum\limits_{n =0}^{L-1}\big\|R_n^4[\Pi_{\le N}u^n]\big\|_{L^2_x}
\lesssim &
\tau^2 \Big(N+\tau N^2\Big)\mathcal Y([0,T])^6.
\end{align*}
This finishes the proof of the lemma.  
\end{proof}

Third, we give the estimates on $R_n^5$.
\begin{lemma}\label{lem:R5}
Let $\gamma>0$,  then there exists  $C>0$ independent of $\tau, N$ and $u^n$, such that 
$$
\sum\limits_{n =0}^{L-1}\big\|R_n^5[\Pi_{\le N}u^n]\big\|_{L^2_x}
\le C \tau^2N^{\frac{7}{2}-\gamma+}  \mathcal Y([0,T])^5.
$$
\end{lemma}
\begin{proof} 
By H\"older and Bernstein's inequalities,   \eqref{est:opr-e} and $\tau N^2\le 1$, 
we have that 
\begin{align*}
\big\|R_n^5[\Pi_{\le N}u^n]\big\|_{L^2}
\lesssim 
&\tau^2 \big\|\fe^{\frac i2\tau\partial_x^2}\Pi_{\le N}u^n\big\|_{L^6_x}^2\\
&\quad\cdot \Big\|\fe^{\frac i2\tau\partial_x^2}\Pi_{\le N}\big(|\Pi_{\le N}u^n|^2 \Pi_{\le N}u^n\big)-\big|\fe^{\frac i2\tau\partial_x^2}\Pi_{\le N}u^n\big|^2\fe^{\frac i2\tau\partial_x^2}\Pi_{\le N}u^n\Big\|_{L^6_x}\\
\lesssim 
&\tau^3 N^{2-\gamma+} N^\frac13 \big\|J^{\gamma-}\Pi_{\le N}u^n\big\|_{L^6_x}^5.
\end{align*}
This  implies that 
\begin{align*}
\sum\limits_{n =0}^{L-1}\big\|R_n^5[\Pi_{\le N}u^n]\big\|_{L^2_x}
\lesssim &
\tau^2N^{\frac{7}{2}-\gamma+} \mathcal Y([0,T])^5.
\end{align*}
This finishes the proof of the lemma.  
\end{proof}

%
%

Fourth, we give the estimates on $R_n^6$.
\begin{lemma}\label{lem:R6}
Let $\gamma>0$,  then there exists  some $C>0$ independent of $\tau, N$ and $u^n$, such that 
$$
\sum\limits_{n =0}^{L-1}\big\|R^6_n[\Pi_{\le N}u^n]\big\|_{L^2(\T)}
\le C\tau^2 N \mathcal Y([0,T])^6.
$$
\end{lemma}
\begin{proof} 
Note that for any $y\in \R$,
$$
|\Psi_1(iy)|\lesssim |y|^3.
$$
Therefore, by Bernstein's inequality, \eqref{est:opr-e},  $\tau N^2\le 1$ and $\|u^n\|_{L^2}\lesssim \|u^0\|_{L^2}$, 
we have that 
\begin{align*}
\big\|R_n^6[\Pi_{\le N}u^n]\big\|_{L^2_x}
\lesssim &
\tau^3\big\|\fe^{\frac i2\tau\partial_x^2}\Pi_{\le N}u^n\big\|_{L^{14}_x}^7\\
\lesssim &
\tau^3N \big\|\fe^{\frac i2\tau\partial_x^2}\Pi_{\le N}u^n\big\|_{L^2_x}\big\|\fe^{\frac i2\tau\partial_x^2}\Pi_{\le N}u^n\big\|_{L^6_x}^6\\
\lesssim &
\tau^3N \big\|\Pi_{\le N}u^n\big\|_{L^6_x}^6.
\end{align*}
Hence, we have that 
\begin{align*}
\sum\limits_{n =0}^{L-1}\big\|R^6_n[\Pi_{\le N}u^n]\big\|_{L^2(\T)}
\lesssim &
\tau^2N \mathcal Y([0,T])^6.
\end{align*}
 thus give the proof of the lemma.
\end{proof}

Collecting the results in Lemmas \ref{lem:R3}--\ref{lem:R6}, we complete the proof of Lemma \ref{lem:Rn1}. 
 

\subsection{Estimates on $\Upsilon_n^2$}

First, we give the estimate on $r^n$ in the following lemma. 
\begin{lemma} \label{lem:rn}
Let  $\gamma>0, T>0, N\ge 1$   and $u_0\in H^\gamma$, then there exist $K_0>0$ and  $C=C(T,\|u_0\|_{H^\gamma})>0$, such that 
\begin{align*}
\big\|r^n\big\|_{l^\frac32_n L^\infty_t L^2_x([t_n,t_{n+1}])}
\le &
C\tau^\frac43 N^{2-\gamma+} \mathcal Y([0,T])^4.
\end{align*}
\end{lemma}
\begin{proof}
We denote 
$$
N_n(t)=-i\lambda (t-t_n)\fe^{-it_n\partial_x^2} \big[|\Pi_{\le N} u^n|^2\Pi_{\le N} u^n\big]
$$
for short. Then by the definitions \eqref{def:rn} and \eqref{V-ltN}, we have that 
\begin{subequations}\label{rn-123}
\begin{align}
r^n(t)=&-i \lambda  \int_{t_n}^t \fe^{-is\partial_x^2}\Pi_{\le N}
    \Big[\Big|\fe^{is\partial_x^2}\Pi_{\le N}v^n\Big|^2\fe^{is\partial_x^2}\Pi_{\le N}v^n\Big]\,ds\notag\\
    &\qquad +i \lambda  \int_{t_n}^t \fe^{-it_n\partial_x^2}\Pi_{\le N}
    \Big[\Big|\fe^{it_n\partial_x^2}\Pi_{\le N}v^n\Big|^2\fe^{it_n\partial_x^2}\Pi_{\le N}v^n\Big]\,ds\label{rn-1}\\
    &-2i \lambda  \int_{t_n}^t \fe^{-is\partial_x^2}\Pi_{\le N}
    \Big[\Big|\fe^{is\partial_x^2}\Pi_{\le N}v^n\Big|^2\fe^{is\partial_x^2}\Pi_{\le N}N_n(s)\Big]\,ds\notag\\
    &\qquad -i \lambda  \int_{t_n}^t \fe^{-is\partial_x^2}\Pi_{\le N}
    \Big[\fe^{-is\partial_x^2}\Pi_{\le N}\overline{v^n}\cdot\Big(\fe^{is\partial_x^2}\Pi_{\le N}N_n(s)\Big)^2\,ds\notag\\
    &\qquad -i \lambda  \int_{t_n}^t \fe^{-is\partial_x^2}\Pi_{\le N}
    \Big[\fe^{-is\partial_x^2}\Pi_{\le N}\overline{N_n(s)}\cdot\Big(\fe^{is\partial_x^2}\Pi_{\le N}\big(v^n+N_n(s)\big)\Big)^2\Big]\,ds\label{rn-2}\\
    &+\frac{t-t_n}{\tau}\Upsilon_n^1.\label{rn-3}
\end{align}
\end{subequations}

For \eqref{rn-1}, similar as the proof of Lemma \ref{lem:R3}, by Lemma \ref{lem:multiplier}, we have that 
\begin{align*}
\big\|\eqref{rn-1}\big\|_{L^\infty_t L^2_x([t_n,t_{n+1}])}
\lesssim \tau^2 N^{2-\gamma+}\big\|J^{\gamma-}\Pi_{\le N}u^n\big\|_{L^6_x}^3.
\end{align*}
Therefore,
\begin{align*}
\big\|\eqref{rn-1}\big\|_{l^\frac32_n L^\infty_t L^2_x([t_n,t_{n+1}])}
\lesssim \tau^\frac43 N^{2-\gamma+}\mathcal Y([0,T])^3.
\end{align*}

For \eqref{rn-2}, note that by Bernstein's inequality and  \eqref{est:opr-e},   it infers that 
\begin{align}\label{N-1}
\big\|\fe^{it\partial_x^2}\Pi_{\le N}N_n(t)\big\|_{L^\infty_t L^6_x([t_n,t_{n+1}])}
\lesssim &
\big\|J^\frac13\fe^{it\partial_x^2}\Pi_{\le N}N_n(t)\big\|_{L^\infty_t L^2_x([t_n,t_{n+1}])}\notag\\
\lesssim &
\tau  \big\|\Pi_{\le N}u^n\big\|_{L^2_x}\big\|J^\frac16\Pi_{\le N}u^n\big\|_{L^6_x}\big\|J^\frac12\Pi_{\le N}u^n\big\|_{L^6_x}\notag\\
\lesssim &
\tau \big\|J^\frac16\Pi_{\le N}u^n\big\|_{L^6_x}\big\|J^\frac12\Pi_{\le N}u^n\big\|_{L^6_x}.
\end{align}
This combining with $\|u^n\|_{L^2}\le \|u^0\|_{L^2}$, further yields that 
\begin{align}\label{N-2}
\big\|\fe^{it\partial_x^2}\Pi_{\le N}N_n(t)\big\|_{L^\infty_t L^6_x([t_n,t_{n+1}])}
\lesssim &
\tau N\big\|\Pi_{\le N}u^n\big\|_{L^6_x}\big\|\Pi_{\le N}u^n\big\|_{L^2_x} 
\lesssim \tau N\big\|\Pi_{\le N}u^n\big\|_{L^6_x}.
\end{align}

Then using \eqref{N-1}, \eqref{N-2},    H\"older and Bernstein's inequalities,  and $2-\gamma\ge 0$, we get that 
\begin{align*}
\big\|\eqref{rn-2}\big\|_{L^\infty_t L^2_x([t_n,t_{n+1}])}
\lesssim &
\tau \big\|\fe^{is\partial_x^2}\Pi_{\le N}N_n(t)\big\|_{L^\infty_t L^6_x}\\
&\quad \cdot
 \Big(\big\|\fe^{is\partial_x^2}\Pi_{\le N}v^n\big\|_{L^\infty_t L^6_x}+\big\|\fe^{is\partial_x^2}\Pi_{\le N}N_n(t)\big\|_{L^\infty_t L^6_x}\Big)^2\\
\lesssim &
\tau^2 \big\|J^\frac16\Pi_{\le N}u^n\big\|_{L^6_x}\big\|J^\frac12\Pi_{\le N}u^n\big\|_{L^6_x}
\Big(\big\|\Pi_{\le N}u^n\big\|_{L^6_x}+ \tau N\big\|\Pi_{\le N}u^n\big\|_{L^6_x}\Big)^2\\
\lesssim &
\tau^2N^{2-\gamma+} \big\|J^{\gamma-}\Pi_{\le N}u^n\big\|_{L^6_x}^4.
\end{align*}
Therefore,
\begin{align*}
\big\|\eqref{rn-2}\big\|_{l^\frac32_n L^\infty_t L^2_x([t_n,t_{n+1}])}
\lesssim \tau^\frac43 N^{2-\gamma+} \mathcal Y([0,T])^4.
\end{align*}

For \eqref{rn-3},  by Lemma \ref{lem:Rn1} and $l^1\hookrightarrow l^\frac32$, we have that 
\begin{align*}
\big\|\eqref{rn-3}\big\|_{l^\frac32_n L^\infty_t L^2_x([t_n,t_{n+1}])}
\le& \big\|\Upsilon_n^1\big\|_{l^\frac32_n H^\gamma_x}\\
\le& \big\|\Upsilon_n^1\big\|_{l^1_n H^\gamma_x}
\lesssim   \tau^2 N^{4-\gamma+}\Big(\mathcal Y([0,T])^3+\mathcal Y([0,T])^6\Big).
\end{align*}

Together with the three estimates on \eqref{rn-123}, we  finish the proof of the lemma. 
\end{proof}

\begin{lemma} \label{lem:Rn2}
Let $\gamma\in (0,2]$,  then there exists $C>0$ independent of $\tau, N$ and $u^n$, such that 
\begin{align*}
\sum\limits_{n =0}^{L-1}\big\|\Upsilon_n^2\big\|_{L^\infty_t L^2_x([t_n,t_{n+1}])}
\lesssim &
\tau^2N^{2-\gamma+} \big(\mathcal Y([0,T])^4+\mathcal Y([0,T])^{12}\big).
\end{align*}
\end{lemma}
\begin{proof}
By \eqref{Rnt}, H\"older and Bernstein's inequalities,
\begin{align*}
\big\|\Upsilon_n^2\big\|_{L^\infty_t L^2_x([t_n,t_{n+1}])}
\lesssim &
\int_{t_n}^{t_{n+1}} \big\|\fe^{is\partial_x^2} r^n\big\|_{L^6_x} \Big(\big\|\fe^{is\partial_x^2} r^n\big\|_{L^6_x}^2+\big\|\mathcal U_{\le N}\big\|_{L^6_x}^2\Big)\,ds\\
\lesssim &
\tau \big\|\fe^{it\partial_x^2} r^n\big\|_{L^\infty_t L^6_x([t_n,t_{n+1}])}^3+\tau^\frac23\big\|\fe^{it\partial_x^2} r^n\big\|_{L^\infty_t L^6_x([t_n,t_{n+1}])}\big\|\mathcal U_{\le N}\big\|_{L^6_{tx}([t_n,t_{n+1}])}^2.
\end{align*}
Therefore,
\begin{align*}
\sum\limits_{n =0}^{L-1}\big\|\Upsilon_n^2\big\|_{L^\infty_t L^2_x([t_n,t_{n+1}])}
\lesssim &
\tau \big\|\fe^{it\partial_x^2} r^n\big\|_{l_n^3 L^\infty_t L^6_x([t_n,t_{n+1}])}^3+\tau^\frac23\big\|\fe^{it\partial_x^2} r^n\big\|_{l^\frac32_n L^\infty_t L^6_x([t_n,t_{n+1}])}\big\|\mathcal U_{\le N}\big\|_{L^6_{tx}([0,T])}^2\\
\lesssim &
\tau^5 N^{6-3\gamma+}  \mathcal Y([0,T])^{12}
+\tau^2 N^{2-\gamma+} \mathcal Y([0,T])^4\cdot \big[\mathcal Y([0,T])+1\big]^2\\
\lesssim &
\tau^2N^{2-\gamma+}  \big(\mathcal Y([0,T])^4+\mathcal Y([0,T])^{12}\big).
\end{align*}
This gives the proof of the lemma.
\end{proof}

\begin{proof}[Proof of Proposition \ref{prop:R}]
By \eqref{def:R}, 
\begin{align*}
\big\|\Upsilon(t)\big\|_{L^\infty_tH^\gamma_x([0,T])}
\lesssim 
\sum\limits_{n =0}^{L-1}\big\|\Upsilon_n^1\big\|_{L^2_x}
+\sum\limits_{n =0}^{L-1}\big\|\Upsilon_n^2\big\|_{L^\infty_t L^2_x([t_n,t_{n+1}])}.
\end{align*}
Then it is followed from Lemmas \ref{lem:Rn1} and \ref{lem:Rn2}.
\end{proof}

\subsection{Estimates on $\Upsilon_3$ and $\Upsilon_4$}
First, we give the estimates on $R_3$. 
\begin{proposition} \label{prop:R34} Let $T>0$ and $L=\lfloor \frac T\tau\rfloor$, there exists  some $C>0$ independent of $\tau, N$ and $u^n$, such that 
$$
\big\|\Upsilon_3\big\|_{L^\infty_tL^2_x([0,T])}+\big\|\Upsilon_4\big\|_{L^\infty_tL^2_x([0,T])}
\le C\Big( \tau^2 \sum\limits_{n =0}^{L-1}\big\|\Pi_{>N}v^n\big\|_{L^2}+N^{-\min\{2\gamma,\alpha(\gamma)\}}+\tau N^{-\gamma+}\Big). 
$$
\end{proposition}
\begin{proof}
For $\Upsilon_3$, note that for any $y\in \R$,
$$
|\Psi_2(iy)|\lesssim |y|^2.
$$
then by \eqref{def:V-btN}, we have that for any $s\in [t_n, t_{n+1}]$, 
$$
\big\|\mathcal V_{>N}(s)-\Pi_{>N}v^n\big\|\lesssim \tau  \big\|\Pi_{>N}v^n\big\|_{L^2}.
$$
Hence, 
\begin{align*}
\big\|\Upsilon^3_n\big\|_{L^\infty_tL^2_x([0,T])}
\lesssim 
\tau \big\|\mathcal V_{>N}(s)-\Pi_{>N}v^n\big\|_{L^\infty_tL^2_x([0,T])}+\tau^2 \big\|\Pi_{>N}v^n\big\|_{L^2}
\lesssim 
\tau^2 \big\|\Pi_{>N}v^n\big\|_{L^2}.
\end{align*}
Therefore,  
$$
\big\|\Upsilon_3\big\|_{L^\infty_tL^2_x([0,T])}
\lesssim \sum\limits_{n =0}^{L-1}\big\|\Upsilon^3_n  \big\|_{L^\infty_tL^2_x([0,T])}
\lesssim \tau^2 \sum\limits_{n =0}^{L-1}\big\|\Pi_{>N}v^n\big\|_{L^2}. 
$$

For $\Upsilon_4$, by Lemma \ref{lem:est-vt-2}, Proposition \ref{prop:R0}, and Bernstein's inequality, we have that 
\begin{align*}
\big\|\Upsilon_4\big\|_{L^\infty_tL^2_x([0,T])}
\lesssim &
\tau N^{-\gamma+}\sum\limits_{j=0}^{n-1}\big\| J^{\gamma-}\big(v(s)-v(t_j)\big)\big\|_{L^\infty_tL^2_x([t_j,t_{j+1}])}+\big\| \Upsilon_0(t)\big\|_{L^\infty_tL^2_x([0,T])}\\
\lesssim 
& N^{-\min\{2\gamma,\alpha(\gamma)\}}+\tau N^{-\gamma+}.
\end{align*} 
This combining the estimate on $\Upsilon_3$ gives the proof of the lemma. 
\end{proof}

\section{Convergence estimates in short time}\label{Convergence}

\subsection{Estimates on $\mathcal Y([0,T])$}
\begin{lemma}\label{lem:YT} Let $\gamma>0, T\in (0,T_0], N\ge 1$ and $u_0\in H^\gamma$, then there exist  $K_0>0$ and $C=C(T,\|u_0\|_{H^\gamma})>0$, such that 
\begin{align*}
\mathcal Y([0,T])\le C\delta_0\big(1+\mathcal X([0,T])+\mathcal X([0,T])^3\big)
+C \big(\mathcal Y([0,T])^3+\mathcal Y([0,T])^{12}\big).
\end{align*}
\end{lemma}
\begin{proof}
By \eqref{U-eqs}, we have that 
\begin{align*}
\Pi_{\le N}u^n
=&\fe^{i(t_n-t)\partial_x^2}\mathcal U_{\le N}(t)
 +i \lambda  \int_{t_n}^{t} \fe^{i(t_n-s)\partial_x^2}\Pi_{\le N}
    \Big[\big|\mathcal U_{\le N}(s)\big|^2\mathcal U_{\le N}(s)\Big]\,ds\\
    &\quad +\fe^{it_n\partial_x^2}\Big(\frac{t-t_n}{\tau} \Upsilon_n^1+\Upsilon_n^2(t)\Big).
\end{align*}
Therefore, by \eqref{est:opr-e}, Lemmas \ref{lem:Stri} and \ref{lem:kato-Ponce}, and Bernstein's inequality,  we have that 
\begin{align*}
\tau^\frac16\big\|J^{\gamma-}\Pi_{\le N}u^n\big\|_{L^6_{x}}
=&\big\|J^{\gamma-}\Pi_{\le N}u^n\big\|_{L^6_{tx}([t_n,t_{n+1}])}\\
\le 
&\big\|J^{\gamma-}\fe^{i(t_n-t)\partial_x^2}\mathcal U_{\le N}(t)\big\|_{L^6_{tx}([t_n,t_{n+1}])}\\
    &\quad
 +\Big\| \int_{t_n}^{t} \fe^{i(t_n-s)\partial_x^2}\Pi_{\le N}J^{\gamma-}
    \Big[\big|\mathcal U_{\le N}(s)\big|^2\mathcal U_{\le N}(s)\Big]\,ds\Big\|_{L^6_{tx}([t_n,t_{n+1}])}\\
    &\quad +\Big\|\fe^{it_n\partial_x^2}J^{\gamma-}\Big(\frac{t-t_n}{\tau} \Upsilon_n^1+\Upsilon_n^2(t)\Big)\Big\|_{L^6_{tx}([t_n,t_{n+1}])}\\
\lesssim &
\big\|J^{\gamma-}\mathcal U_{\le N}(t)\big\|_{L^6_{tx}([t_n,t_{n+1}])}
 +\big\|J^{\gamma}\mathcal U_{\le N}\big\|_{L^4_{tx}([t_n,t_{n+1}])}^3\\
&\quad+  \tau^\frac16 N^{\frac13+\gamma-} \Big(\big\|\Upsilon_n^1\big\|_{L^2_x}
+\big\|\Upsilon_n^2\big\|_{L^\infty_t L^2_x([t_n,t_{n+1}])}\Big).
\end{align*}
Therefore,  by  Lemmas  \ref{lem:JU}, \ref{lem:Rn1} and \ref{lem:Rn2}, we have that  
\begin{align*}
\tau^\frac16\big\|J^{\gamma-}\Pi_{\le N}u^n\big\|_{l^6_nL^6_{x}}
\lesssim &
\big\|J^{\gamma-}\mathcal U_{\le N}(t)\big\|_{L^6_{tx}([0,T])}
 +\big\|J^{\gamma}\mathcal U_{\le N}\big\|_{L^4_{tx}([0,T])}^3\\
&\quad+  \tau^\frac16 N^{\frac13+\gamma-} \sum\limits_{n =0}^{L-1}\Big(\big\|\Upsilon_n^1\big\|_{L^2_x}
+\big\|\Upsilon_n^2\big\|_{L^\infty_t L^2_x([t_n,t_{n+1}])}\Big)\\
\lesssim &
\delta_0\big(1+\mathcal X([0,T])+\delta_0^2\mathcal X([0,T])^3\big)
+\big(\mathcal Y([0,T])^3+\mathcal Y([0,T])^{12}\big).
\end{align*}
This finishes the proof of the lemma.
\end{proof}

\subsection{Estimates on $\mathcal X([0,T])$}
\begin{lemma} \label{lem:XT}
Let $\gamma>0, T\in (0,T_0], N\ge 1$ and $u_0\in H^\gamma$. Denote that 
$$
X_0\triangleq A(N,\tau)^{-1}\big\|\mathcal E_{\le N}(0)\big\|_{L^2},
$$
then there exists  some $C=C(T,\|u_0\|_{H^\gamma})>0$, such that 
$$
\mathcal X([0,T])\le C\Big[1+X_0+
  A(N,\tau)^2\mathcal X([0,T])^3+\Big(\mathcal Y([0,T])^3+\mathcal Y([0,T])^{12}\Big)\Big].
$$
Moreover, 
\begin{align}\label{Energy}
\big\|\fe^{-iT\partial_x^2}&\mathcal E_{\le N}(T)-\big(\mathcal E_{\le N}(0)-\Pi_{\le N}\Upsilon_0(T)\big)\big\|_{L^2_x}\notag\\
\le & C A(N,\tau)\Big[
  \mathcal X([0,T])\Big(A(N,\tau)^2\mathcal X([0,T])^2+\delta_0^2 \Big)+\Big(\mathcal Y([0,T])^3+\mathcal Y([0,T])^{12}\Big)\Big].
\end{align}
\end{lemma}
\begin{proof}
By \eqref{def-E},  \eqref{solution-est-small}, Lemma \ref{lem:Stri},  Propositions \ref{prop:R0} and \ref{prop:R},  for any $T\in (0, T_0]$, we have that 
\begin{align*}
&\big\|\mathcal E_{\le N}\big\|_{L^\infty_tL^2_x\cap L^4_{tx}([0,T])}+\big\|J^{0-}\mathcal E_{\le N}\big\|_{L^6_{tx}([0,T])}\\
\lesssim 
&\big\| \mathcal E_{\le N}(0)\big\|_{L^2}+ \Big\|\Pi_{\le N}
    \Big[\big|\mathcal U_{\le N}(s)\big|^2\mathcal U_{\le N}(s)-\big|\Pi_{\le N}u(s)\big|^2\Pi_{\le N}u(s)\Big]\Big\|_{L^\frac43_{tx}([0,T])}\notag\\
    &+\big\|\Upsilon(t)\big\|_{L^\infty_tL^2_x([0,T])}+\big\|\Upsilon_0(t)\big\|_{L^\infty_tL^2_x([0,T])}\\
\lesssim &
\big\| \mathcal E_{\le N}(0)\big\|_{L^2}+ \big\|\mathcal E_{\le N}\big\|_{L^4_{tx}[0,T]}\Big(\big\|\mathcal E_{\le N}\big\|_{L^4_{tx}[0,T]}^2+\|u\|_{L^4_{tx}[0,T]}^2 \Big)\\
&\qquad+C  \tau^2 N^{4-\gamma+}\Big(\mathcal Y([0,T])^3+\mathcal Y([0,T])^{12}\Big)+C\Big(N^{-\min\{2\gamma,\alpha(\gamma)\}}+\tau N^{-\gamma+}\Big)\\
\lesssim &
 A(N,\tau)\Big[1+X_0+
  \mathcal X([0,T])\Big(A(N,\tau)^2\mathcal X([0,T])^2+\delta_0^2 \Big)+\Big(\mathcal Y([0,T])^3+\mathcal Y([0,T])^{12}\Big)\Big].
\end{align*} 
Therefore, we obtain that 
$$
\mathcal X([0,T])\lesssim 1+X_0+
  \mathcal X([0,T])\Big(A(N,\tau)^2\mathcal X([0,T])^2+\delta_0^2 \Big)+\Big(\mathcal Y([0,T])^3+\mathcal Y([0,T])^{12}\Big).
$$
Choosing $\delta_0$ suitably small, then we obtain that 
$$
\mathcal X([0,T])\le C\Big[1+X_0+
  A(N,\tau)^2\mathcal X([0,T])^3+\Big(\mathcal Y([0,T])^3+\mathcal Y([0,T])^{12}\Big)\Big].
$$
Moreover, note that  
\begin{align*}
\big\|\fe^{-iT\partial_x^2}&\mathcal E_{\le N}(T)-\big(\mathcal E_{\le N}(0)-\Pi_{\le N}\Upsilon_0(t)\big)\big\|_{L^2_x}\notag\\
\le &  \Big\|\Pi_{\le N}
    \Big[\big|\mathcal U_{\le N}(s)\big|^2\mathcal U_{\le N}(s)-\big|\Pi_{\le N}u(s)\big|^2\Pi_{\le N}u(s)\Big]\Big\|_{L^\frac43_{tx}([0,T])}+\big\|\Upsilon(t)\big\|_{L^\infty_tL^2_x([0,T])},
\end{align*}
then \eqref{Energy} is followed by the same estimates above. 
\end{proof}

Combining  the estimates in Lemmas  \ref{lem:YT} and \ref{lem:XT}, we have that for any $T\in (0,T_0]$, there exists $C_1,C_2>0$, such that
\begin{equation}\label{est:XY-T}
\begin{split}
\mathcal Y([0,T])\le  &C_1\Big[\delta_0\big(1+\mathcal X([0,T])+\delta_0^2\mathcal X([0,T])^3\big)
+\big(\mathcal Y([0,T])^3+\mathcal Y([0,T])^{12}\big)\Big];\\
\mathcal X([0,T])\le & C_2 \Big[1+X_0+
  A(N,\tau)^2\mathcal X([0,T])^3+\Big(\mathcal Y([0,T])^3+\mathcal Y([0,T])^{12}\Big)\Big].
  \end{split}
\end{equation}
Since $X_0=0, A(N,\tau)\le N^{-\gamma}$,  by bootstrap and choosing $\tau_0$ suitably small, we obtain that 
for any $\tau \in (0,\tau_0]$ and any $T\in (0,T_0]$, 
\begin{align}\label{XY-Tj}
\mathcal X([0,T])\le 2C_2;\mbox{ and } \mathcal Y([0,T])\le 4C_1C_2 \delta_0. 
\end{align}

\section{Iteration and proof of Proposition \ref{prop:main}}\label{proof-of-Theorem}
We split the convergence estimates into two parts:  the low-frequency and the high-frequency components. 

\subsection{Convergence estimate: low-frequency component}
Now we extend the estimates \eqref{XY-T} to any fixed $T^*<\infty$. From Lemma \ref{lem:global-theory}, we have that 
there exists some constant $C(T,\|u_0\|_{H^\gamma})> 0$, such that 
\begin{align*}
\big\|J^\gamma u\big\|_{L^4_{tx}([0,T^*]\times \T)}
+\big\|J^{\gamma-} u\big\|_{L^6_{tx}([0,T^*]\times \T)}
\le C(T^*,\|u_0\|_{H^\gamma}).
\end{align*}
Fixing $\delta_0>0$, then there exists $K\sim \frac{T^*}{\delta_0}$ such that 
$$
[0,T^*]=[0,T_1]\cup  [T_1,T_2]\cup\cdots  \cup [T_{K-1},T_K],
$$
with 
\begin{align}\label{solution-est-small-j}
\big\|J^\gamma u\big\|_{L^4_{tx}([T_k,T_{k+1}])}
+\big\|J^{\gamma-} u\big\|_{L^6_{tx}([T_k,T_{k+1}])}
\le \delta_0, \quad \mbox{for each}\quad k=0,\cdots, K-1.
\end{align}

Denote that 
$$
X_k\triangleq A(N,\tau)^{-1}\big\|\mathcal E_{\le N}(T_k)\big\|_{L^2},
$$
then replacing $0$ and $T$ by $T_k$ and $T_{k+1}$ in \eqref{est:XY-T}, we have that for any $T\in [T_k,T_{k+1}]$, 
\begin{equation}\label{est:XY-Tk}
\begin{split}
\mathcal Y([T_k,T])\le  &C_1\Big[\delta_0\big(1+\mathcal X([T_k,T])+\delta_0^2\mathcal X[T_k,T]^3\big)
+\big(\mathcal Y([T_k,T])^3+\mathcal Y([T_k,T])^{12}\big)\Big];\\
\mathcal X([T_k,T])\le & C_2 \Big[1+X_k+
  A(N,\tau)^2\mathcal X([T_k,T])^3+\Big(\mathcal Y([T_k,T])^3+\mathcal Y([T_k,T])^{12}\Big)\Big].
  \end{split}
\end{equation}
Moreover, 
\begin{align}
\big\|\fe^{-iT_{k+1}\partial_x^2}\mathcal E_{\le N}(T_{k+1})&-\fe^{-iT_k\partial_x^2}\big(\mathcal E_{\le N}(T_k)+\Upsilon_0(T_{k+1})\big)\big\|_{L^2}\notag\\
\le  C A(N,\tau)\Big[&
  \mathcal X([T_k,T_{k+1}])\Big(A(N,\tau)^2\mathcal X([T_k,T_{k+1}])^2+\delta_0^2 \Big)\notag\\
  &+\Big(\mathcal Y([T_k,T_{k+1}])^3+\mathcal Y([T_k,T_{k+1}])^{12}\Big)\Big].\label{Energy-k}
\end{align}
%

Suppose that  there exists some $K_0\in \Z: 0\le K_0\le K-1$, such that for any $k=0,1,\cdots, K_0$,
\begin{align}\label{ASS-k}
X_k\le C_0,
\end{align}
where $C_0$ can be decided by the inequality from Proposition \ref{prop:R0}: 
$$\big\|\Upsilon_0(t)\big\|_{L^\infty_tL^2_x([0,T^*])}
\le \frac12 C_0 A(N,\tau).
$$
Then by \eqref{est:XY-Tk} and the bootstrap,  and choosing $\tau_0$ suitably small, we obtain that  
for any $\tau \in (0,\tau_0]$ and any $T\in (T_k,T_{k+1}]$, 
\begin{align}\label{XY-T}
\mathcal X([T_k,T_{k+1}])\le 3C_2;\mbox{ and } \mathcal Y([T_k,T_{k+1}])\le 4C_1C_2 \delta_0. 
\end{align}
Therefore, by \eqref{Energy-k},  Proposition \ref{prop:R0}   and \eqref{XY-T}, we obtain that 
\begin{align*}
\big\|\mathcal E_{\le N}(T_{K_0+1})\big\|_{L^2}
=&
\big\|\fe^{-iT_{K_0+1}\partial_x^2}\mathcal E_{\le N}(T_{K_0+1})\big\|_{L^2}\\
\le &\big\|\Upsilon_0(T_{K_0}\big)\big\|_{L^2_x}+  C A(N,\tau)\sum\limits_{k=0}^{K_0}\Big[
  \mathcal X([T_k,T_{k+1}])\Big(A(N,\tau)^2\mathcal X([T_k,T_{k+1}])^2+\delta_0^2 \Big)\\
  &+\Big(\mathcal Y([T_k,T_{k+1}])^3+\mathcal Y([T_k,T_{k+1}])^{12}\Big)\Big]\\
   \le  &
     A(N,\tau) \Big[\frac12 C_0+ C_3K_0\big(A(N,\tau)^2+\delta_0^2+\delta_0^3+\delta_0^{12}\big)\Big]. 
\end{align*}
where $C_3=C_3(C_1,C_2)>0$.  
Since $K\delta_0\sim T^*$, then by \eqref{tauN2-small}, it further infers that 
\begin{align}
\big\|\mathcal E_{\le N}(T_{K_0+1})\big\|_{L^2}
   \le  &
     A(N,\tau) \Big[\frac12 C_0+ C_4\delta_0\Big],
\end{align}
where $C_4=C_4(T^*,C_3)>0$. Choosing $\delta_0$ suitably small, then we get that 
\begin{align*}
\big\|\mathcal E_{\le N}(T_{K_0+1})\big\|_{L^2}
\le C_0
     A(N,\tau). 
\end{align*}
Thus, \eqref{ASS-k} holds for $k=0,1,\cdots, K_0+1$. Then by iteration, we obtain that  \eqref{ASS-k} holds for $k=0,1,\cdots, K$. 
Hence, \eqref{XY-T} holds for $k=0,1,\cdots, K-1$. 
In particular, this implies that 
\begin{align}\label{Low-fre-conv}
\max_{1\le n\le L}  \|\Pi_{\le N}u(t_n)-\Pi_{\le N}u^{n}\|_{L^2}
  \le C\Big( N^{-\min\{2\gamma,\alpha(\gamma)\}}+\tau N^{-\gamma+}+\tau^2 N^{4-\gamma}\Big).
\end{align}

\subsection{Convergence estimate: high-frequency component}
 
 First, since $\|\Pi_{>N}v^n\|_{L^2}\le \|v^0\|_{L^2}$, by Proposition \ref{prop:R34}, we have that
 $$
\big\|\Upsilon_3\big\|_{L^\infty_tL^2_x([0,T])}+\big\|\Upsilon_4\big\|_{L^\infty_tL^2_x([0,T])}
\le C\Big( \tau +N^{-\min\{2\gamma,\alpha(\gamma)\}}\Big). 
$$
Hence, from \eqref{E-high}, we obtain that for any $T>0$, 
$$
\big\|\mathcal E_{> N}\big\|_{L^\infty_tL^2_x([0,T])}\le C\Big(\tau+N^{-\min\{2\gamma,\alpha(\gamma)\}}\Big). 
$$
In particular, this together with Lemma \ref{lem:global-theory} implies that 
\begin{align*}
\big\|\Pi_{>N}v^n\big\|_{L^2}\le& \big\|\Pi_{>N}\mathcal E_{> N}(t_n)\big\|_{L^2}+\big\|\Pi_{>N}v(t_n)\big\|_{L^2}\\
\lesssim &
 \tau +N^{-\min\{2\gamma,\alpha(\gamma)\}}
 +N^{-\gamma}.
\end{align*}
Inserting this estimate into Proposition \ref{prop:R34}, we further have that
 $$
\big\|\Upsilon_3\big\|_{L^\infty_tL^2_x([0,T])}+\big\|\Upsilon_4\big\|_{L^\infty_tL^2_x([0,T])}
\le C\Big( \tau^2 +N^{-\min\{2\gamma,\alpha(\gamma)\}}+\tau N^{-\gamma+}\Big). 
$$
Hence, by \eqref{E-high} again, we obtain that for any $T>0$, 
$$
\big\|\mathcal E_{> N}\big\|_{L^\infty_tL^2_x([0,T])}\le C\Big( \tau^2 +N^{-\min\{2\gamma,\alpha(\gamma)\}}+\tau N^{-\gamma+}\Big). 
$$
In particular, this implies that 
\begin{align}\label{High-fre-conv}
\max_{1\le n\le L}  \|\Pi_{> N}u(t_n)-\Pi_{> N}u^{n}\|_{L^2}
  \le C\Big( \tau^2 +N^{-\min\{2\gamma,\alpha(\gamma)\}}+\tau N^{-\gamma+}\Big).
\end{align}

Now combining with the estimates \eqref{Low-fre-conv} and \eqref{High-fre-conv}, we finish the proof of Proposition \ref{prop:main}
\subsection{From \eqref{NuSo-NLS-1} to \eqref{NuSo-NLS-2}}

Now we give the sketch (the details are omitted since it can be followed by the same manner) how we obtain \eqref{NuSo-NLS-2} from \eqref{NuSo-NLS-1}.
First, we use the following formulas:
\begin{align*}
\Pi_{>N}\big[\fe^{i\tau\partial_x^2-4\pi i\lambda M_0\tau}u\big]
&=\fe^{i\frac\tau2\partial_x^2}\fe^{-i\lambda \tau \big|\fe^{i\frac\tau2\partial_x^2}\Pi_{\le N}u\big|^2-2\pi i\lambda M_0\tau}\Pi_{>N}\fe^{i\frac\tau2\partial_x^2}u\\
&\quad -\fe^{i\frac\tau2\partial_x^2}\left[\fe^{-i\lambda \tau\big|\fe^{i\frac\tau2\partial_x^2}\Pi_{\le N}u\big|^2+2\pi i\lambda M_0\tau}-1\right]\fe^{-4\pi i\lambda M_0\tau}\Pi_{>N}\fe^{i\frac\tau2\partial_x^2}u.
\end{align*}
Note that 
\begin{align*}
 &-\fe^{i\frac\tau2\partial_x^2}\left[\fe^{-i\lambda \tau\big|\fe^{i\frac\tau2\partial_x^2}\Pi_{\le N}u\big|^2+2\pi i\lambda M_0\tau}-1\right]\fe^{-4\pi i\lambda M_0\tau}\Pi_{>N}\fe^{i\frac\tau2\partial_x^2}u\\
 \approx &i\lambda \tau \fe^{-4\pi i\lambda M_0\tau}\fe^{i\frac\tau2\partial_x^2}\left[\big|\fe^{i\frac\tau2\partial_x^2}\Pi_{\le N}u\big|^2-2\pi i\lambda M_0\right]\Pi_{>N}\fe^{i\frac\tau2\partial_x^2}u\\
  \approx &i\lambda  \fe^{-4\pi i\lambda M_0\tau}\int_0^\tau\fe^{is\partial_x^2}\left[\big|\fe^{is\partial_x^2}\Pi_{\le N}u\big|^2-2\pi i\lambda M_0\right]\Pi_{>N}\fe^{is\partial_x^2}u.
\end{align*}
Then from Lemma \ref{lem:tri-est}, it has a nice control as  $N^{-\min\{2\gamma,\alpha(\gamma)\}}$. Therefore, we can replace 
$$
\Pi_{>N}\big[\fe^{i\tau\partial_x^2-4\pi i\lambda M_0\tau}u\big]
$$
by 
\begin{align}\label{P1}
\fe^{i\frac\tau2\partial_x^2}\fe^{-i\lambda \tau \big|\fe^{i\frac\tau2\partial_x^2}\Pi_{\le N}u\big|^2-2\pi i\lambda M_0\tau}\Pi_{>N}\fe^{i\frac\tau2\partial_x^2}u,
\end{align}
Furthermore, we can drop the first filter $\Pi_\tau$ in the low frequency component, that is, replacing 
$$
\Pi_\tau\fe^{i\frac\tau2\partial_x^2}\mathcal N_\tau\big[\fe^{i\frac\tau2\partial_x^2}\Pi_\tau u^n\big]
$$
by 
\begin{align}\label{P2}
\fe^{i\frac\tau2\partial_x^2}\mathcal N_\tau\big[\fe^{i\frac\tau2\partial_x^2}\Pi_\tau u^n\big].
\end{align}
See \eqref{R1-a} and Proposition \ref{prop:R0} for its estimate. 
Therefore, combining with \eqref{P1} and \eqref{P2}, we  get \eqref{NuSo-NLS-1}. 

%

\vskip .4in
\section*{Acknowledgements}
The author was partially supported by the NSFC (No. 12171356).
The author would also like to express his deep gratitude to  Professor Enrique Zuazua for his very helpful private discussion, and many valuable suggestions.

\vskip 25pt

\bibliographystyle{model1-num-names}

\begin{thebibliography}{00}

\bibitem{besse}
C. Besse, B. Bid\'{e}garay, and S. Descombes: 
Order estimates in time of splitting methods
for the nonlinear Schr\"odinger equation. 
{\it SIAM J. Numer. Anal.} 40 (2002), pp. 26--40.

\bibitem{Bo}
{J. Bourgain}: 
Fourier transform restriction phenomena for certain lattice subsets and
applications to nonlinear evolution equations. I. Schr\"odinger  equations. 
{\it Geom. Funct. Anal.} 3 (1993), pp. 107--156.
 
%
\bibitem{BoLi-KatoPonce}
{J. Bourgain and D. Li}: 
On an endpoint Kato-Ponce inequality. 
{\it Differential Integral Equations} 27 (2014), pp. 1037--1072.

%
%
\bibitem{ESS-2016}
J. Eilinghoff, R. Schnaubelt, and K. Schratz: 
Fractional error estimates of splitting schemes for the nonlinear Schrödinger equation. 
{\it J. Math. Anal. Appl.} 442 (2016), pp. 740--760.

%
%
%
\bibitem{Ignat-2011}
L. I. Ignat: A splitting method for the nonlinear Schrödinger equation.
{\it J. Differential Equations} 250 (2011), pp. 3022--3046.

\bibitem{IgnatZua-2009}
L.I. Ignat and E. Zuazua:
Numerical dispersive schemes for the nonlinear Schr\"odinger equation. 
{\it SIAM J. Numer. Anal.}  47 (2009), pp. 1366--1390.

%
\bibitem{Kato-Ponce}
{T. Kato and G. Ponce}: 
Commutator estimates and the Euler and Navier-Stokes equations. 
{\it Commun. Pure Appl. Math.} 41 (1988) pp. 891-907.

\bibitem{ORS-2022}
A. Ostermann, F. Rousset, K. Schratz: 
Error estimates at low regularity of splitting schemes for NLS, to appear in Math. Comp., 2022. 

\bibitem{Li-KatoPonce}
{D. Li}: On Kato-Ponce and fractional Leibniz. 
{\it Rev. Mat. Iberoam.} 35 (2019) pp. 23--100.

\bibitem{Lubich-2008}
Ch. Lubich: 
On splitting methods for Schr\"odinger-Poisson and cubic nonlinear Schr\"odinger equations. {\it Math. Comp.} 77 (2008), pp. 2141--2153.

\bibitem{LiWu-2022}
Buyang Li, Yifei Wu: 
An unfiltered low-regularity integrator for the KdV equation with solutions below $H^1$, arXiv:2206.09320.



\bibitem{Strang-1968}
G. Strang, On the construction and comparison of difference schemes, SIAM J. Numer. Anal. 5 (1968), 506--517. 

%
%
%
%
%
%
%
%
%
%
%
%
\bibitem{Lpmulitiplier-mathoverflow}
\url{https://mathoverflow.net/questions/390335}


\end{thebibliography}

\end{document}